\documentclass[a4paper, reqno]{amsart}

\usepackage{enumitem} %For choosing how to enumerate things

\usepackage{bm} %Used for bold math
\usepackage{amsmath} %Used for introducing notation
\usepackage{amssymb} %Used for various predefined commands
\usepackage{hyperref} %Used for making references hyperlink
\usepackage{tikz} %Used for diagrams
\usepackage{mathtools} %Used for absolute value, coloneqqs
\usepackage{extarrows} %Used for longer Longleftrightarrow
\usetikzlibrary{cd} %Used for tikzcd

\newcommand{\K}{\mathbf{k}}
\newcommand{\La}{\Lambda}
\newcommand{\cC}{\mathcal{C}}
\newcommand{\cD}{\mathcal{D}}
\newcommand{\cP}{\mathcal{P}}
\newcommand{\cI}{\mathcal{I}}
\newcommand{\divides}{\mid}
\newcommand{\ZZ}{\mathbb{Z}}
\newcommand{\isom}{\cong}
\newcommand{\J}{\mathcal{J}}

\DeclareMathOperator{\add}{add}
\DeclareMathOperator{\rad}{rad}
\DeclareMathOperator{\Ext}{Ext}
\DeclareMathOperator{\Hom}{Hom}
\DeclareMathOperator{\m}{mod}
\DeclareMathOperator{\coker}{coker}
\DeclareMathOperator{\soc}{soc}
\DeclareMathOperator{\op}{op}
\DeclareMathOperator{\lcm}{lcm}

\DeclareRobustCommand\longtwoheadrightarrow
     {\relbar\joinrel\twoheadrightarrow}
\newcommand{\longhookrightarrow}{\lhook\joinrel\longrightarrow}
     
\providecommand{\abs}[1]{\lvert#1\rvert}

\tikzset{
  symbol/.style={
    draw=none,
    every to/.append style={
      edge node={node [sloped, allow upside down, auto=false]{$#1$}}}
  }
}%To have symbols in tikzcd arrows

\theoremstyle{plain}
\newtheorem*{theorem*}{Theorem}
\newtheorem{theorem}{Theorem}[section] % reset theorem numbering for each chapter

\theoremstyle{plain} % just in case the style had changed
\newcommand{\thistheoremname}{}
\newtheorem*{genericthm*}{\thistheoremname}
\newenvironment{namedthm*}[1]
  {\renewcommand{\thistheoremname}{#1}%
   \begin{genericthm*}}
  {\end{genericthm*}}

\theoremstyle{definition}
\newtheorem{definition}[theorem]{Definition} % definition numbers are dependent on theorem numbers
\newtheorem{definition-proposition}[theorem]{Definition-Proposition} % definition numbers are dependent on theorem numbers
\newtheorem{example}[theorem]{Example} % same for example numbers 
\newtheorem{corollary}[theorem]{Corollary} % same for corollary numbers 
\newtheorem{lemma}[theorem]{Lemma} % same for lemma numbers 
\newtheorem{proposition}[theorem]{Proposition} % same for proposition numbers
\newtheorem{remark}[theorem]{Remark} % remarks are the same
\numberwithin{equation}{section} % equations are numbered separately, only with respect to their section
\title[$n$-cluster tilting subcategories for radical square zero algebras]{$n$-cluster tilting subcategories for radical square zero algebras}
\author{Laertis Vaso}


\begin{document}

\maketitle

\begin{abstract}
    We give a characterization of radical square zero bound quiver algebras $\K Q/\J^2$ that admit $n$-cluster tilting subcategories and $n\ZZ$-cluster tilting subcategories in terms of $Q$. We also show that if $Q$ is not of cyclically oriented extended Dynkin type $\tilde{A}$, then the poset of $n$-cluster tilting subcategories of $\K Q/\J^2$ with relation given by inclusion forms a lattice isomorphic to the opposite of the lattice of divisors of an integer which depends on $Q$.
\end{abstract}

\section*{Introduction}

Representation theory of algebras can be described as the study of the category $\m\La$ of fi\-nite-di\-men\-sion\-al (right) modules over an algebra $\La$. One of the most helpful tools in that study  has been Aus\-lan\-der--Rei\-ten theory. In recent years a high\-er-di\-men\-sion\-al analogue of Aus\-lan\-der--Rei\-ten theory has been introduced by Iyama \cite{IYA2, IYA4}; see also  \cite{IYA1}. In this theory, instead of focusing on $\m\La$, one restricts to a suitable subcategory $\cC$ of $\m\La$ called an \emph{$n$-cluster tilting subcategory} for some positive $n$, while if $\cC$ has an additive generator $M$, then $M$ is called an \emph{$n$-cluster tilting module}. In this setting one may describe $\cC$ using an $n$-dimensional version of Aus\-lan\-der--Rei\-ten theory.

Every algebra $\La$ admits a unique $1$-cluster tilting subcategory, namely $\m\La$ itself. On the other hand, if $n\geq 2$, then an $n$-cluster tilting subcategory may not exist. Generally, it is not easy to find algebras which admit $n$-cluster tilting subcategories. Recently there has been a lot of research in trying to find or construct $n$-cluster tilting subcategories, see for example \cite{IO2,HI,IO,CIM,JK,CDIM}.

For simplicity, we assume that all quivers in this article are connected; the results of this paper can be straightforwardly generalised for quivers which are not connected. For an integer $m\in\ZZ_{\geq 1}$ we denote by $A_m$ the quiver $1\longrightarrow 2 \longrightarrow \cdots \longrightarrow m$ and by $\tilde{A}_m$ the quiver
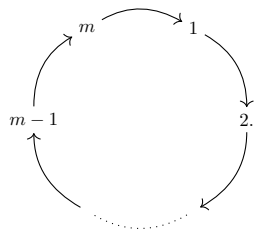
\[
\begin{tikzpicture}[scale=0.7, transform shape]
\node (1) at (1,1.732050808) {$1$};
\node (0) at (-1,1.732050808) {$m$};
\node (m-1) at (-2,0) {$m-1$};
\node (m-2) at (-1,-1.732050808) {$ $};
\node (dots) at (1,-1.732050808) {$ $};
\node (2) at (2,0) {$2$.};
\draw[->] (0) to [out=30, in=150]  node[draw=none, above] {} (1);
\draw[->] (1) to [out=-30, in=90] node[draw=none, midway, right] {} (2);
\draw[->] (m-1) to [out=90, in=-150] node[draw=none, midway, left] {} (0);
\draw[->] (m-2) to [out=150, in=-90] node[draw=none, midway, left] {} (m-1);
\draw[->] (2) to [out=-90, in=30] node[draw=none, midway, right] {} (dots);
\draw[dotted] (dots) to [out=-150,in=-30] (m-2);
\end{tikzpicture}
\]
For a quiver $Q$ we denote by $\J$ the ideal of the path algebra $\K Q$ generated by the arrows of $Q$. One may then ask for which integers $l\geq 2$ does the bound quiver algebra $\La=\K Q/\J^l$ admit an $n$-cluster tilting subcategory. Several results are known in that direction. The case $Q=A_m$ has been studied in \cite{VAS}, while the case $Q=\tilde{A}_m$ has been studied in \cite{DI}. The case where $n$ is the global dimension of $\La$ was studied in \cite{ST}. In this article we consider the case where $l=2$. As a first result we have the following theorem.

\begin{namedthm*}{Theorem A}[Proposition \ref{prop:Q is n-pre-admissible} and Theorem \ref{thrm:representation-finite string algebra}]\label{thrm:A}
Let $\La=\K Q/\J^2$ and $n\geq 2$. If $\La$ admits an $n$-cluster tilting subcategory, then $\La$ is a rep\-re\-sen\-ta\-tion-fi\-nite string algebra. Moreover, if $X$ is an indecomposable $\La$-module and $X$ is not simple, then $X$ is projective or injective.
\end{namedthm*}

Theorem A shows that a radical square zero bound quiver algebra $\La$ which admits an $n$-cluster tilting subcategory is well-understood from the point of view of representation theory. In particular, since $\La$ is rep\-re\-sen\-ta\-tion-fi\-nite, every $n$-cluster tilting subcategory $\cC$ of $\m\La$ is of the form $\cC=\add(M)$ for an $n$-cluster tilting module $M\in\m\La$. We then give the following characterization, which is the main result of this paper.

\begin{namedthm*}{Theorem B}[Theorem \ref{thrm:n-cluster tilting iff n-admissible}]\label{thrm:B} 
Let $\La=\K Q/J^2$ and $n\geq 2$. Then $\La$ admits an $n$-cluster tilting subcategory $\cC$ if and only if $Q$ is an $n$-admissible quiver. If moreover $Q\neq \tilde{A}_m$, then $\cC$ is unique and $\cC=\add\left(\bigoplus_{j\geq 0}\tau_n^{-j}(\La)\right)$ where $\tau_n^{-}=\tau^{-}\Omega^{-(n-1)}$.
\end{namedthm*}

For the definition of $n$-admissible quivers we refer to Definition \ref{def:n-pre-admissible} and Definition \ref{def:n-admissible}; we refer to Remark \ref{rem:how to find n-admissible quivers} for an easy way to construct $n$-admissible quivers. Given Theorem B, it is not hard to classify radical square zero bound quiver algebras which admit $n\ZZ$-cluster tilting subcategories in the sense of \cite{IJ}. 

\begin{namedthm*}{Theorem C}[Theorem \ref{thrm:nZ-cluster tilting iff Nakayama}]
Let $\La=\K Q/\J^2$ and $n\geq 2$. Then $\La$ admits an $n\ZZ$-cluster tilting subcategory if and only if $Q=A_m$ and $n\divides (m-1)$ or $Q=\tilde{A}_m$ and $n\divides m$.
\end{namedthm*}

Finally, we show that if $Q\neq\tilde{A}_m$, then the set of $n$-cluster tilting subcategories of $\K Q/\J^2$ forms a lattice isomorphic to the lattice of a certain integer which depends only on $Q$.

\begin{namedthm*}{Theorem D}[Theorem \ref{thrm:lattice of n-ct}]
Let $\La=\K Q/\J^2$. Assume that $Q\neq \tilde{A}_m$ and that $Q$ has admissible degree $N$. Set
\[
    \mathbf{CT}(\La)\coloneqq\{\cC\subseteq \m\La \mid \text{there exists $n\in\ZZ_{\geq 1}$ such that $\cC$ is $n$-cluster tilting}\}.
\]
Then $(\mathbf{CT}(\La),\subseteq)$ is a complete lattice isomorphic to the opposite of the lattice of divisors of $N$.
\end{namedthm*}

For the definition of the admissible degree of a quiver we refer to Definition \ref{def:admissible degree}.

This paper is organized as follows. In Section \ref{sec:preliminaries and notation} we establish notation and include some general results about $n$-cluster tilting subcategories and radical square zero algebras. In Section \ref{sec:necessary conditions} we find some necessary conditions for a radical square zero bound quiver algebra to admit an $n$-cluster tilting subcategories. In Section \ref{sec:sufficient conditions} we show that these necessary conditions are also sufficient. In Section \ref{sec:main result and applications} we state our main result and a few applications.

\section{Preliminaries and notation}\label{sec:preliminaries and notation}

Let $\K$ be a field. By an algebra we mean a fi\-nite-di\-men\-sion\-al associative $\K$-algebra with a unit and by a module we mean a fi\-nite-di\-men\-sion\-al right module. 

Let $\La$ be an algebra. The \emph{(Jacobson) radical} $\rad(\La)$ of $\La$ is the intersection of all the maximal right ideals of $\Lambda$. The algebra $\La$ is a \emph{radical square zero} algebra if $\rad^2(\La)=0$. We denote by $\m\La$ the category of $\La$-modules. A $\La$-module $M\in\m\La$ is called \emph{basic} if all indecomposable direct summands of $M$ are pairwise non-isomorphic. For $M\in \m\La$ we denote by $\Omega(M)$ the \emph{syzygy} of $M$, that is the kernel of $P(M)\longtwoheadrightarrow M$, where $P(M)$ is the \emph{projective cover} of $M$ and by $\Omega^{-}(M)$ the \emph{cosyzygy} of $M$, that is the cokernel of $M\longhookrightarrow I(M)$ where $I(M)$ is the \emph{injective hull} of $M$. Note that $\Omega(M)$ and $\Omega^{-}(M)$ are unique up to isomorphism. 

We denote by $D$ the duality $\Hom(-,K)$ between $\m\La$ and $\m\La^{\op}$. We denote by $\tau$ and $\tau^-$ the \emph{Aus\-lan\-der--Rei\-ten translations} and we recall the \emph{Aus\-lan\-der--Rei\-ten duality}
\[
\Ext^{1}_{\La}(M,N)\isom D\underline{\Hom}_{\La}(\tau^{-}(N),M),
\]
for all $M,N\in\m\La$, where $\underline{\Hom}_{\La}(-,-)$ denotes morphisms in the \emph{projectively stable category} $\underline{\m}\La$. For more details about the representation theory of fi\-nite-di\-men\-sion\-al algebras and Aus\-lan\-der--Rei\-ten theory we refer to \cite{ARS, ASS}.

Throughout this article $n$ denotes a positive integer. A subcategory $\cC\subseteq \m\La$ is called \emph{$n$-rigid} if $\Ext^{i}_{\La}(\cC,\cC)=0$ for all $i\in\{1,\ldots,n-1\}$. A functorially finite subcategory $\cC\subseteq\m\La$ is called \emph{$n$-cluster tilting} if
\begin{align*}
    \cC &= \{ X \in \m\La \mid \Ext^{i}_{\La}\left(X, \cC\right)=0\text{ for all $0 < i < n$}\} \\
    & = \{ X \in \m\La \mid \Ext^{i}_{\La}\left(\cC,X\right)=0\text{ for all $0 < i < n$}\}.
\end{align*}
If moreover $\cC=\add(M)$ for a module $M\in\m\La$, then $M$ is called an \emph{$n$-cluster tilting module}. Notice that any category of the form $\add(M)$ for some $M\in\m\La$ is functorially finite. In particular, if $\La$ is rep\-re\-sen\-ta\-tion-fi\-nite, then any subcategory of $\m\La$ is functorially finite. Clearly if $\cC\subseteq \m\La$ is $n$-cluster tilting, then $\La,D(\La)\in\cC$; we use this fact throughout. We denote by $\tau_n$ and $\tau_n^{-}$ the \emph{$n$-Aus\-lan\-der--Rei\-ten translations} defined by $\tau_n=\tau\Omega^{n-1}$ and $\tau_n^{-}=\tau^{-}\Omega^{-(n-1)}$. For more details about higher dimensional Aus\-lan\-der--Rei\-ten theory we refer to \cite{IYA1}.

Notice that there exists a unique $1$-cluster tilting subcategory of $\m\La$, namely $\m\La$ itself. In the rest of this paper we assume that $n\geq 2$, unless otherwise stated. We also need the following observations.

\begin{proposition}\label{prop:n-ct is closed under n-AR translations and uniqueness of n-ct}
Let $\La$ be a fi\-nite-di\-men\-sion\-al algebra and let $\cC\subseteq\m\La$ be an $n$-cluster tilting subcategory. 
\begin{enumerate}[label=(\alph*)]
    \item The functors $\tau_n:\cC_{\cP}\to \cC_{\cI}$ and $\tau_n^{-}:\cC_{\cI}\to \cC_{\cP}$ induce mutually inverse bijections, between the set $\cC_{\cP}$ of isomorphism classes of indecomposable nonprojective $\La$-modules and the set $\cC_{\cI}$ of isomorphism classes of indecomposable noninjective $\La$-modules.
    
    \item If $\cD\subseteq\m\La$ is an $n$-cluster tilting subcategory such that $\cD\subseteq \cC$, then $\cC=\cD$.
    
    \item Let $M=\bigoplus_{j\geq 0}\tau_n^{-j}(\La)$. Then $M\in\cC$. If moreover $M$ is an $n$-cluster tilting module, then $\cC=\add(M)$.
\end{enumerate}
\end{proposition}

\begin{proof}
\begin{enumerate}[label=(\alph*)]
    \item See \cite[Theorem 2.8]{IYA1}.
    
    \item Follows directly from the definition of $n$-cluster tilting subcategories.
    
    \item Since $\La\in\cC$, we have that $M\in\cC$ by (a). In particular we have $\add(M)\subseteq \cC$. Hence if $M$ is an $n$-cluster tilting module, then by (b) we conclude that $\cC=\add(M)$.\qedhere
\end{enumerate}
\end{proof}

\begin{lemma}\label{lem:taun-inverse gives nonzero ext}
Let $\La$ be a fi\-nite-di\-men\-sion\-al algebra and $M,N\in\m\La$ with $M\neq 0$. Assume that $\tau_x^{-}(N)\isom M$ for some $x\geq 1$. Then $\Ext_{\La}^{x}(M,N)\neq 0$.
\end{lemma}

\begin{proof}
We first consider the case $x=1$. By additivity of $\tau^{-}$ and $\Ext_{\La}^{x}(-,-)$ we may assume that $M$ and $N$ are indecomposable. Since $\tau^{-}(N)\isom M$ and $M$ is nonzero, it follows that $N$ is noninjective. Then there exists an almost split sequence $0\longrightarrow N\longrightarrow F\longrightarrow \tau^{-}(N)\longrightarrow 0$ in $\m\La$ and the result follows. For $x\geq 2$ we have using dimension shift that
\[
\Ext_{\La}^{x}(M,N) \isom \Ext_{\La}^{1}(\tau_x^{-}(N),\Omega^{-(x-1)}(N)) = \Ext_{\La}^{1}(\tau^{-}(\Omega^{-(x-1)}(N)),\Omega^{-(x-1)}(N)) \neq 0, 
\]
where the last inequality follows from the case $x=1$.
\end{proof}

Next we recall some background on bound quiver algebras. A \emph{quiver} $Q=(Q_0,Q_1,s,t)$ is a quadruple consisting of a set $Q_0$ of \emph{vertices}, a set $Q_1$ of \emph{arrows} and two maps $s,t:Q_1\to Q_0$ called \emph{source map} and \emph{target map}. All quivers in this article are \emph{finite}, that is both $Q_0$ and $Q_1$ are finite sets. Moreover, for simplicity, we assume that all quivers in this article are \emph{connected}, that is the underlying unoriented graph of $Q$ is connected. For a vertex $v\in Q_0$, the \emph{ingoing degree} of $v$, denoted by $\delta^{-}(v)$, is the number of arrows ending at $v$ and the \emph{outgoing degree} of $v$, denoted by $\delta^{+}(v)$, is the number of arrows starting at $v$. The \emph{degree} of $v$ is the tuple $(\delta^{-}(v),\delta^{+}(v))$. For a quiver $Q$ and $k\geq 1$, a \emph{path $\mathbf{p}$ of length $k$} in $Q$ is a sequence of $k$ consecutive arrows
\[\mathbf{p} = 
    \begin{tikzcd} 
        v_1 \arrow[r, "\alpha_1"] & v_2 \arrow[r, "\alpha_2"] & \cdots \arrow[r] & v_{k} \arrow[r, "\alpha_{k}"] & v_{k+1},
    \end{tikzcd}
\]
in $Q$. We also assign a trivial path $\epsilon_v$ of length $0$ to each vertex $v\in Q_0$. 

Let $Q$ be a quiver. We denote by $\K Q$ the \emph{path algebra} of $Q$ and we denote by $\J\subseteq \K Q$ the \emph{arrow ideal} of $Q$, that is the ideal of $\K Q$ generated by the arrows of $Q$. An ideal $\cI\subseteq\K Q$ is called \emph{admissible} if there exists $k\geq 2$ such that $\J^k \subseteq \cI \subseteq \J^2$. If $\cI$ is an admissible ideal, then the \emph{bound quiver algebra }$\La=\K Q/\cI$ is a fi\-nite-di\-men\-sion\-al algebra. Throughout we identify $\La$-modules and representations of $Q$ bound by $\cI$. For a vertex $v\in Q_0$, we denote by $P(v)$, $I(v)$ and $S(v)$ the indecomposable projective, injective and simple $\La$-modules corresponding to $v$. When clear from context, we use composition series to denote $\La$-modules. For more details on bound quiver algebras and their representation theory we refer to \cite{ARS, ASS}.

For radical square zero algebras we have the following easy observations.

\begin{lemma}\label{lem:semisimple syzygy} 
Let $\La$ be a radical square zero algebra and let $M$ be a nonprojective $\La$-module. Then $\Omega(M)$ is semisimple.
\end{lemma}

\begin{proof}
Since $M$ is nonprojective, it follows that $\Omega(M)\neq 0$. Let $P(M)$ be the projective cover of $M$. Then $\rad^2(P(M))=P(M)\rad^2(\La)=0$ and so $\rad(P(M))$ is semisimple. Since $\Omega(M)$ is a submodule of $\rad(P(M))$ and $\Omega(M)\neq 0$, we conclude that $\Omega(M)$ is semisimple.
\end{proof}

\begin{lemma}\label{lem:dimension of syzygy of ind.inj.}
Let $\La$ be a radical square zero algebra and assume that $\cC\subseteq \m\La$ is an $n$-cluster tilting subcategory. Let $I$ be an indecomposable injective $\La$-module. Then $\dim(\Omega(I))\leq 1$.
\end{lemma}

\begin{proof}
If $I$ is projective, then $\dim(\Omega(I))=0$. Otherwise, assume that $I$ is nonprojective. By Lemma \ref{lem:semisimple syzygy} we have that $\Omega(I)$ is semisimple. By \cite[Corollary 3.3]{VAS} and since $I\in\cC_{\cP}$, we have that $\Omega(I)$ is indecomposable. Since $\Omega(I)$ is semisimple and indecomposable, it follows that $\dim(\Omega(I))\leq 1$.
\end{proof}

In this paper we study radical square zero bound quiver algebras. These can be easily described as in the following lemma.

\begin{lemma}\label{lem:radical square zero bound quiver algebra}
A bound quiver algebra $\K Q/ \cI$ is a radical square zero algebra if and only if $\cI=\J^2$.
\end{lemma}

\begin{proof}
Since $\cI$ is admissible, we have that $\cI=\J^2$ if and only if $\J^2\subseteq \cI$, which is equivalent to the ideal $(\J/\cI)^2$ being equal to the zero ideal. Since $(\J/\cI)^2=\rad^2(\K Q/\cI)$, the result follows.
\end{proof}

As a corollary, any radical square zero algebra over an algebraically closed field is Morita equivalent to a bound quiver algebra of the form $\K Q/\J^2$.

\begin{proposition}\label{prop:radical square zero over algebraically closed field}
Let $\La$ be a basic and connected fi\-nite-di\-men\-sion\-al $\K$-algebra and assume that $\K$ is algebraically closed. Then $\La$ is a radical square zero algebra if and only if $\La\isom \K Q/\J^2$ for some quiver $Q$.
\end{proposition}

\begin{proof}
Since $\La$ is basic and $\K$ is algebraically closed, there exists a quiver $Q$ and an admissible ideal $\cI\subseteq \K Q$ such that $\La\isom \K Q/\cI$. The result follows from Lemma \ref{lem:radical square zero bound quiver algebra}.
\end{proof}

We also need to recall the following notion.

\begin{definition}\label{def:string algebra}
A bound quiver algebra $\K Q/\cI$ is a \emph{string algebra} if the following conditions hold:
\begin{enumerate}
    \item[(S1)] For every vertex $v\in Q_0$ we have that $\delta^{-}(v)\leq 2$ and $\delta^{+}(v)\leq 2$.
    \item[(S2)] For every arrow $\alpha\in Q_1$ there exists at most one arrow $\beta\in Q_1$ such that $\beta\alpha\not\in\cI$ and at most one arrow $\gamma\in Q_1$ such that $\alpha\gamma\not\in\cI$.
    \item[(S3)] The ideal $\cI$ can be generated by paths.
\end{enumerate}
\end{definition}

Indecomposable modules over string algebras are classified in \cite{ButRin} using the combinatorics of strings and bands. We briefly recall these combinatorics. 

Let $\K Q/\cI$ be a string algebra. For every arrow $\alpha\in Q_1$ we define a formal inverse $\alpha^{-}$ such that $s(\alpha^{-})=t(\alpha)$ and $t(\alpha^{-})=s(\alpha)$. We define $Q_1^{-}=\{\alpha^{-}\mid \alpha\in Q_1\}$ and we set $(\alpha^{-})^{-}=\alpha$. We call elements of $Q_1$ \emph{direct arrows} and elements of $Q_1^{-}$ \emph{inverse arrows}. A \emph{formal path of length $k\geq 1$} is a sequence $\bm{\ell}=\ell_k\ldots\ell_1$ such that $\ell_i\in Q_1\cup Q_1^{-}$ and such that for all $i\in\{1,\ldots,k-1\}$ we have $t(\ell_i)=s(\ell_{i+1})$ and $\ell_i\neq \ell_{i+1}^{-}$. We also set $\bm{\ell}^{-}=\ell_1^{-}\ldots\ell_k^{-}$. We say that $\bm{\ell}$ is a \emph{string of length $k$} if no formal path of the form $\ell_{i+r}\ldots\ell_{i}$ or $\ell_{i}^{-}\ldots\ell_{i+r}^{-}$ is in $\cI$ for $1\leq i\leq i+r \leq k$. To each vertex $v\in Q_0$ we also associate a string $\bm{e_v}$ of length $0$. We say that a string $\bm{\ell}$ is a \emph{band} if $s(\ell_1)=t(\ell_k)$ and $\bm{\ell}^q$ is a string for every $q\geq 1$, and moreover there is no string $\bm{\ell'}\neq\bm{\ell}$ such that $\bm{\ell'}\ldots\bm{\ell'}=\bm{\ell}$. For each string or band $\bm{\ell}$ we can define a corresponding string or band module $M(\bm{\ell})$ which is indecomposable. Furthermore, every indecomposable $\K Q/\cI$-module is isomorphic to a string or band module. For more details on the definition of $M(\bm{\ell})$ and on other facts about the representation theory of string algebras we refer to \cite[Section 3]{ButRin}.

\section{Necessary conditions}\label{sec:necessary conditions}

In this section we investigate the existence of an $n$-cluster tilting subcategory $\cC\subseteq \m\La$ where $\La=\K Q/\J^2$ and $n\geq 2$. Our aim is to show that these assumptions impose some important restrictions on $Q$ and $\La$. 

\subsection{\texorpdfstring{$n$}{n}-pre-admissible quivers}\label{subsec:n-preadimssible}

Recall that if $\cC\subseteq \m\La$ is an $n$-cluster tilting subcategory, then $\La,D(\La)\in\cC$. We start with showing that the degree of a vertex in $Q$ is bounded.

\begin{lemma}\label{lem:bound on number of arrows}
Let $\La=\K Q/\J^2$ and assume that $\cC\subseteq \m\La$ is an $n$-cluster tilting subcategory. Let $v\in Q_0$ be a vertex. Then $\delta^{-}(v)\leq 2$ and $\delta^{+}(v)\leq 2$.
\end{lemma}

\begin{proof}
We only show that $\delta^{-}(v)\leq 2$; the inequality $\delta^{+}(v)\leq 2$ follows dually. Consider the short exact sequence $0\longrightarrow \Omega(I(v)) \longrightarrow P(I(v)) \longrightarrow I(v) \longrightarrow 0$. Then 
\begin{equation}\label{eq:dim of syzygy of injective}
    \dim(I(v)) = \dim(P(I(v))) - \dim(\Omega(I(v))) \geq \dim(P(I(v)))-1,
\end{equation}
where the last inequality follows from Lemma \ref{lem:dimension of syzygy of ind.inj.}. Moreover, since $\La$ is a radical square zero bound quiver algebra, it immediately follows that $\dim(I(v))=\delta^{-}(v)+1$ and $\dim(P(I(v))) \geq 2\delta^{-}(v)$. Hence (\ref{eq:dim of syzygy of injective}) gives
\[\delta^{-}(v)+1 \geq 2\delta^{-}(v)-1,\]
or equivalently $2\geq \delta^{-}(v)$.
\end{proof}

We continue with showing that there are no multiple arrows between two vertices of $Q$.

\begin{lemma}\label{lem:no multiple arrows}
Let $\La=\K Q/\J^2$ and assume that $\cC\subseteq \m\La$ is an $n$-cluster tilting subcategory. Let $v,u\in Q_0$ be vertices. Then $\abs{\{\alpha\in Q_1 \mid s(\alpha)=v \text{ and } t(\alpha)=u\}}\leq 1$.
\end{lemma}

\begin{proof}
By Lemma \ref{lem:bound on number of arrows} we have that $\abs{\{\alpha\in Q_1 \mid s(\alpha)=v, t(\alpha)=u\}}\leq 2$. Assume towards a contradiction that there exist two arrows $\alpha_1:v\longrightarrow u$ and $\alpha_2:v\longrightarrow u$. Then, by Lemma \ref{lem:bound on number of arrows} we have that the composition series of $I(u)$ is $\begin{smallmatrix} v\;\;v \\ u \end{smallmatrix}$ while the composition series of $P(v)$ is $\begin{smallmatrix} v \\ u\;\;u \end{smallmatrix}$. Hence the projective cover of $I(u)$ is $P(I(u))\isom P(v)\oplus P(v)$ and
\begin{align*}
    \dim(\Omega(I(u))) &= \dim(P(I(u))) - \dim(I(u)) =2\dim(P(v)) - \dim(I(u)) =2\cdot 3-3 =3,
\end{align*}
which contradicts Lemma \ref{lem:dimension of syzygy of ind.inj.}.
\end{proof}

Next we show that no vertex can have degree $(0,2)$ or $(2,0)$ and, moreover, that if $n>2$, then no vertex can have degree $(2,2)$ either.

\begin{lemma}\label{lem:conditions for two arrows}
Let $\La=\K Q/\J^2$ and assume that $\cC\subseteq \m\La$ is an $n$-cluster tilting subcategory. Let $v\in Q_0$ be a vertex. 
\begin{enumerate}[label=(\alph*)]
        \item If $\delta^{+}(v)=2$, then $\delta^{-}(v)\geq 1$.
        \item If $\delta^{-}(v)=2$, then $\delta^{+}(v)\geq 1$.
        \item If $\delta^{-}(v)=\delta^{+}(v)=2$, then $n=2$.
    \end{enumerate}
\end{lemma}

\begin{proof}
\begin{enumerate}[label=(\alph*)]
    \item Since $\La=\K Q/\J^2$ and $\delta^{+}(v)=2$, it follows that $\dim(P(v))=3$. Assume towards a contradiction that $\delta^{-}(v)=0$. Then $I(v)=S(v)$ and $P(I(v))=P(v)$. Hence 
    \[\dim(\Omega(I(v))) = \dim(P(I(v))) - \dim(I(v)) = 3-1 = 2,\]
    which contradicts Lemma \ref{lem:dimension of syzygy of ind.inj.}.
    
    \item Dual to (a).
    
    \item Let $\alpha_1:v\longrightarrow u_1$, $\alpha_2:v\longrightarrow u_2$, $\beta_1:w_1\longrightarrow v$ and $\beta_2:w_2\longrightarrow v$ be the arrows starting and ending at $v$. Then the composition series of $P(v)$ is $\begin{smallmatrix} v \\ u_1\;\;u_2\end{smallmatrix}$ while the composition series of $I(v)$ is $\begin{smallmatrix} w_1\;\; w_2 \\ v \end{smallmatrix}$. For $i=1,2$, let $\pi_i: P(v) \longrightarrow \begin{smallmatrix} v \\ \;\;u_i\end{smallmatrix}$ be the projective cover of $\begin{smallmatrix} v \\ \;\;u_i \end{smallmatrix}$ and $\iota_i:\begin{smallmatrix} w_i \\ v \end{smallmatrix}\longrightarrow I(v)$ be the injective envelope of $\begin{smallmatrix} w_i \\ v \end{smallmatrix}$. Then it follows that 
    \[\coker\left(\begin{bmatrix} \pi_1 \\ \pi_2 \end{bmatrix}:P(v) \longrightarrow \begin{smallmatrix} v \\ \;\;u_1\end{smallmatrix} \oplus \begin{smallmatrix} v\\ \;\;u_2\end{smallmatrix}\right) \isom S(v) \isom \ker\left(\begin{bmatrix} \iota_1 & \iota_2 \end{bmatrix}: \begin{smallmatrix} w_1 \\ v \end{smallmatrix} \oplus \begin{smallmatrix} w_2 \\ v \end{smallmatrix} \longrightarrow I(v)\right).
    \]
    Hence the sequence
    \[\begin{tikzcd}[column sep=small, row sep=small, ampersand replacement=\&]
        0 \arrow[rr] \&\&
        P(v) \arrow[rr, "\begin{bmatrix} \pi_1 \\ \pi_2 \end{bmatrix}"] \&\& 
        \begin{smallmatrix} v \\ \;\;u_1\end{smallmatrix} \oplus \begin{smallmatrix} v\\ \;\;u_2\end{smallmatrix} \arrow[rr] \arrow[rd, two heads] \&\& 
        \begin{smallmatrix} w_1 \\ v \end{smallmatrix} \oplus \begin{smallmatrix} w_2 \\ v \end{smallmatrix} \arrow[rr, "\begin{bmatrix} \iota_1\;\;\; \iota_2 \end{bmatrix}"] \&\&
        I(v) \arrow[rr] \&\&
        0 \\
        {} \&\& 
        {} \&\& 
        {} \&
        S(v) \arrow[ru, hook]
        \& 
        {} \&\&
        {} \&\&
        {}
    \end{tikzcd}
    \]
    gives a nonzero element of $\Ext^2_{\La}(I(v),P(v))$. Since $I(v),P(v)\in\cC$ and $\Ext^2_{\La}(I(v),P(v))\neq 0$, it follows that $n\leq 2$. Since by assumption we have that $n\geq 2$, we conclude that $n=2$.\qedhere
\end{enumerate}
\end{proof}

Finally we examine how an arrow between two vertices affects the degree of the two vertices.

\begin{lemma}\label{lem:no zigzags}
Let $\La=\K Q/\J^2$ and assume that $\cC\subseteq \m\La$ is an $n$-cluster tilting subcategory.
\begin{enumerate}[label=(\alph*)]
    \item Let $w_1\longrightarrow v \longleftarrow w_2$ be a subquiver of $Q$. Then $\delta^{+}(w_1)=\delta^{+}(w_2)=1$.
    \item Let $u_1\longleftarrow v \longrightarrow u_2$ be a subquiver of $Q$. Then $\delta^{-}(u_1)=\delta^{-}(u_2)=1$.
\end{enumerate}
\end{lemma}

\begin{proof}
We only prove (a); (b) follows dually. By symmetry it is enough to show that $\delta^{+}(w_1)=1$. Since by Lemma \ref{lem:bound on number of arrows} we have that $\delta^{+}(w_1)\leq 2$ and since by assumption we have that $\delta^{+}(w_1)\geq 1$, it is enough to show that $\delta^{+}(w_1)\neq 2$. Assume towards a contradiction that $\delta^{+}(w_1)=2$. Since $\La=\K Q/\J^2$, it follows from Lemma \ref{lem:bound on number of arrows} that $\dim(I(v))=3$, $\dim(P(w_1))=3$, $\dim(P(w_2))\geq 2$ and $P(I(v))\isom P(w_1)\oplus P(w_2)$. Then 
\[\dim(\Omega(I(v)))= \dim(P(I(v))) - \dim(I(v)) = \dim(P(w_1))+\dim(P(w_2))-\dim(I(v)) \geq 3+2-3 = 2,\]
which contradicts Lemma \ref{lem:dimension of syzygy of ind.inj.}.
\end{proof}

\begin{corollary}\label{cor:arrow bounds sum of degrees}
Let $\La=\K Q/\J^2$ and assume that $\cC\subseteq\m\La$ is an $n$-cluster tilting subcategory. Let $v\longrightarrow u$ be an arrow in $Q$. Then $\delta^{+}(v)+\delta^{-}(u)\leq 3$.
\end{corollary}

\begin{proof}
Follows immediately by Lemma \ref{lem:bound on number of arrows} and Lemma \ref{lem:no zigzags}.
\end{proof}

The results of this section motivate the following definition.

\begin{definition}\label{def:n-pre-admissible}
A quiver $Q$ is called \emph{$n$-pre-admissible} if the following conditions are satisfied.
    \begin{enumerate}[label=(\roman*)]
        \item For all vertices $v\in Q_0$ we have $\delta(v)\in \{(0,0), (0,1), (1,0), (1,1), (1,2), (2,1)\}\cup E$, where $E=\begin{cases}\{(2,2)\}, &\mbox{\text{if $n=2$,}} \\ \varnothing, &\mbox{\text{otherwise.}} \end{cases}$
        \item There exist no multiple arrows between two vertices.
        \item For all arrows $v\longrightarrow u$ in $Q$ we have $\delta^{+}(v)+\delta^{-}(u)\leq 3$.
    \end{enumerate} 
\end{definition}

\begin{remark}\label{rem:n-pre-admissible is n-pre-admissible for any n}
It follows immediately by the definition of $n$-pre-admissible quivers that an $n$-pre-admissible quiver which has no vertex of degree $(2,2)$ is $n$-pre-admissible for any $n\geq 2$. 
\end{remark}

\begin{example}\label{ex:n-pre-admissible quivers}
\begin{enumerate}[label=(\alph*)]
    \item The quivers $A_m$ and $\tilde{A}_m$ are $n$-pre-admissible for any $n\geq 2$.
    \item The quiver
    \[\begin{tikzcd}[row sep=small, column sep=small]
        1 \arrow[r]  & 2 \arrow[r]\arrow[loop above] & 3 
	\end{tikzcd}\]
	is not $n$-pre-admissible for any $n\geq 2$ since there exists an arrow $2\longrightarrow 2$, but $\delta^{+}(2)+\delta^{-}(2)=4$.
    \item The quiver
    
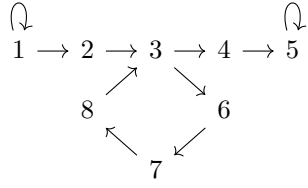
\[\begin{tikzcd}[row sep=small, column sep=small]
        1 \arrow[r] \arrow[loop above] & 2 \arrow[r] & 3 \arrow[r] \arrow[rd] & 4 \arrow[r] & 5 \arrow[loop above]\\
        {} & 8 \arrow[ru] & {} & 6 \arrow[ld] & {} \\
        {} & {} & 7 \arrow[lu] & {} & {}
	\end{tikzcd}\]
	is $2$-pre-admissible but not $n$-pre-admissible for $n\geq 3$ since $\delta(3)=(2,2)$.
	\item The quiver
	
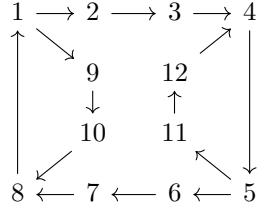
\[\begin{tikzcd}[row sep=small, column sep=small]
        1 \arrow[r] \arrow[rd] & 2\arrow[r] & 3 \arrow[r] & 4 \arrow[ddd] \\
        {} & 9 \arrow[d] & 12 \arrow[ru] & {} \\
        {} & 10 \arrow[ld] & 11 \arrow[u] & {} \\
        8 \arrow[uuu] & 7\arrow[l] & 6 \arrow[l] & 5 \arrow[l] \arrow[lu] 
	  \end{tikzcd}\]
	   is $n$-pre-admissible for any $n\geq 2$. 
\end{enumerate}
\end{example}

We have the following immediate result.

\begin{proposition}\label{prop:Q is n-pre-admissible}
Let $\La=\K Q/\J^2$ and assume that $\cC\subseteq \m\La$ is an $n$-cluster tilting subcategory. Then $Q$ is $n$-pre-admissible.
\end{proposition}

\begin{proof}
Follows immediately by Lemma \ref{lem:bound on number of arrows}, Lemma \ref{lem:conditions for two arrows}, Lemma \ref{lem:no multiple arrows} and Corollary \ref{cor:arrow bounds sum of degrees}.
\end{proof}

Radical square zero bound quiver algebras with $n$-pre-admissible quivers are especially easy to study from the point of view of representation theory. Indeed, we have the following result.

\begin{theorem}\label{thrm:representation-finite string algebra}
Let $Q$ be an $n$-pre-admissible quiver and let $\La=\K Q/\J^2$.
\begin{enumerate}[label=(\alph*)]
    \item $\La$ is a string algebra.
    \item If $\ell_k\ldots \ell_1$ is a string in $\La$, then $k\leq 2$. In particular, there are no bands in $\La$.
    \item $\La$ is rep\-re\-sen\-ta\-tion-fi\-nite.
    \item If $M$ is an indecomposable $\La$-module and $M$ is not simple, then $M$ is projective or injective.
\end{enumerate}
\end{theorem}

\begin{proof}
\begin{enumerate}[label=(\alph*)]
    \item Since $Q$ is $n$-pre-admissible, we have that $\delta^{+}(v)\leq 2$ and $\delta^{-}(v)\leq 2$ for every vertex $v\in Q_0$. Since $\J^2$ is generated by all paths of length $2$, it immediately follows that $\La$ is a string algebra.
    
    \item Let $\ell_k\ldots \ell_1$ be a string in $\La$ and assume towards a contradiction that $k\geq 3$. Consider the string $\ell_3\ell_2\ell_1$. Since every path of length two is in $\J^2$, it follows that $\ell_1$ and $\ell_2$ cannot be both direct or both inverse letters. Similarly $\ell_3$ and $\ell_2$ cannot be both direct or both inverse letters. Hence $\ell_3\ell_2\ell_1$ is either of the form $\alpha\beta^{-}\gamma$ or of the form $\alpha^{-}\beta\gamma^{-}$ for some arrows $\alpha,\beta,\gamma\in Q_1$ with $\alpha\neq \beta$ and $\gamma\neq \beta$. If $\alpha=\gamma$, then we readily get that $s(\alpha)=s(\beta)$ and $t(\alpha)=t(\beta)$, which contradicts Definition \ref{def:n-pre-admissible}(ii). Otherwise, if $\alpha\neq\gamma$, then we readily get that
    \[
    \delta^{+}(s(\beta))+\delta^{-}(t(\beta))\geq 4,
    \]
    which contradicts Definition \ref{def:n-pre-admissible}(iii). Hence $k\leq 2$. Since the length of a string is bounded by $2$, it follows that there are no bands in $\La$.
    
    \item Follows immediately from (b) since indecomposable $\La$-modules are classified by string and band modules, see \cite[Section 3]{ButRin}.
    
    \item Let $M$ be an indecomposable $\La$-module and assume that $M$ is not simple. From (b) it follows that $M$ is isomorphic to a string module $M(\bm{\ell})$ where $\bm{\ell}$ has length at most $2$ (for the definition of $M(\bm{\ell})$ we refer to \cite[Section 3]{ButRin}). Since $M$ is not simple, it follows that $\bm{\ell}$ has length different than $0$ and so $\bm{\ell}$ has length $1$ or $2$. If the length of $\bm{\ell}$ is $1$, then $\bm{\ell}=\alpha$ for some arrow $\alpha\in Q_1$ (the modules $M(\alpha)$ and $M(\alpha^{-})$ are isomorphic). Let $\alpha:v\longrightarrow u$. Then $\delta^{+}(v)\in\{1,2\}$ and using Definition \ref{def:n-pre-admissible}(iii) and the fact that $\La$ is a radical square zero algebra it is easy to see that
    \[
    M(\alpha) \isom \begin{cases} P(v), &\mbox{if $\delta^{+}(v)=1$,} \\ I(u), &\mbox{if $\delta^{+}(v)=2$.} \end{cases}
    \]
    If the length of $\bm{\ell}$ is $2$, and since $\La$ is a radical square zero algebra, then $\bm{\ell}=\alpha\beta^{-}$ or $\bm{\ell}=\alpha^{-}\beta$ for some arrows $\alpha,\beta\in Q_1$. Let $\alpha:v\longrightarrow u$. Similarly to the case of length $1$, it is easy to see that
    \[
    M(\bm{\ell})\isom \begin{cases} P(v), &\mbox{if $\bm{\ell}=\alpha\beta^{-}$,} \\
    I(u), &\mbox{if $\bm{\ell}=\alpha^{-}\beta,$}\end{cases}
    \]
    and so in both cases $M$ is projective or injective.\qedhere
\end{enumerate}
\end{proof}

\begin{example}\label{ex:AR-quivers of n-pre-admissible quivers}
\begin{enumerate}[label=(\alph*)]
    \item Let $Q$ be as in Example \ref{ex:n-pre-admissible quivers}(c). The Aus\-lan\-der--Rei\-ten quiver of $\K Q/\J^2$ is
    \[
	    \begin{tikzpicture}[scale=0.9, transform shape, baseline={(current bounding box.center)}]
        \tikzstyle{nct}=[circle, minimum width=0.6cm, draw, inner sep=0pt, text centered, scale=0.9]
        \tikzstyle{nout}=[circle, minimum width=6pt, draw=none, inner sep=0pt, scale=0.9]
        
        \node[nout] (110) at (7,0.7) {$\begin{smallmatrix} 1 \end{smallmatrix}$};
        \node[nout] (112) at (8.4, 0.7) {$\begin{smallmatrix} 1 \\ 2\;\; \end{smallmatrix}$};
        
        \node[nout] (211) at (7.7,1.4) {$\begin{smallmatrix} 1 \\ 2\;\; 1\end{smallmatrix}$};
        \node[nout] (213) at (9.1,1.4) {$\begin{smallmatrix} 1 \end{smallmatrix}$};
        
        \node[nout] (34) at (2.8,2.1) {$\begin{smallmatrix} 6 \end{smallmatrix}$};
        \node[nout] (36) at (4.2,2.1) {$\begin{smallmatrix} 3 \\ 4\;\; \end{smallmatrix}$};
        \node[nout] (38) at (5.7,2.1) {$\begin{smallmatrix} \;\;8 \\ 3 \end{smallmatrix}$};
        \node[nout] (310) at (7,2.1) {$\begin{smallmatrix} 2 \end{smallmatrix}$};
        \node[nout] (312) at (8.4,2.1) {$\begin{smallmatrix} 1 \\ \;\; 1 \end{smallmatrix}$};
        
        \node[nout] (45) at (3.5,2.8) {$\begin{smallmatrix} 3 \\ 4\;\;6 \end{smallmatrix}$};
        \node[nout] (47) at (4.9,2.8) {$\begin{smallmatrix} 3 \end{smallmatrix}$};
        \node[nout] (49) at (6.3,2.8) {$\begin{smallmatrix} 2\;\;8 \\ 3 \end{smallmatrix}$};
        
        \node[nout] (52) at (1.4,3.5) {$\begin{smallmatrix} \;\;5 \\ 5  \end{smallmatrix}$};
        \node[nout] (54) at (2.8,3.5) {$\begin{smallmatrix} 4 \end{smallmatrix}$};
        \node[nout] (56) at (4.2,3.5) {$\begin{smallmatrix} 3 \\ \;\; 6 \end{smallmatrix}$};
        \node[nout] (58) at (5.6,3.5) {$\begin{smallmatrix} 2 \;\; \\  3 \end{smallmatrix}$};
        \node[nout] (510) at (7,3.5) {$\begin{smallmatrix} 8 \end{smallmatrix}$};
        \node[nout] (512) at (8.4,3.5) {$\begin{smallmatrix} 7 \end{smallmatrix}$};
        \node[nout] (514) at (9.8,3.5) {$\begin{smallmatrix} 6. \end{smallmatrix}$};
        
        \node[nout] (61) at (0.7,4.2) {$\begin{smallmatrix} 5 \end{smallmatrix}$};
        \node[nout] (63) at (2.1,4.2) {$\begin{smallmatrix} 4 \;\; 5 \\ 5 \end{smallmatrix}$};
        \node[nout] (611) at (7.7,4.2) {$\begin{smallmatrix} 7 \\ 8 \end{smallmatrix}$};
        \node[nout] (613) at (9.1,4.2) {$\begin{smallmatrix} 6 \\ 7 \end{smallmatrix}$};
        
        \node[nout] (72) at (1.4,4.9) {$\begin{smallmatrix} 4\;\; \\ 5 \end{smallmatrix}$};
        \node[nout] (74) at (2.8,4.9) {$\begin{smallmatrix} 5 \end{smallmatrix}$};
        
        \draw[->] (61) to (72);
        \draw[->] (61) to (52);
        \draw[->] (72) to (63);
        \draw[->] (52) to (63);
        \draw[->] (63) to (74);
        \draw[->] (63) to (54);
        \draw[->] (54) to (45);
        \draw[->] (34) to (45);
        \draw[->] (45) to (56);
        \draw[->] (45) to (36);
        \draw[->] (56) to (47);
        \draw[->] (36) to (47);
        \draw[->] (47) to (58);
        \draw[->] (47) to (38);
        \draw[->] (58) to (49);
        \draw[->] (38) to (49);
        \draw[->] (49) to (510);
        \draw[->] (49) to (310);
        \draw[->] (510) to (611);
        \draw[->] (110) to (211);
        \draw[->] (310) to (211);
        \draw[->] (611) to (512);
        \draw[->] (211) to (312);
        \draw[->] (211) to (112);
        \draw[->] (512) to (613);
        \draw[->] (312) to (213);
        \draw[->] (112) to (213);
        \draw[->] (613) to (514);
        
        \draw[loosely dotted] (61.east) -- (63);
        \draw[loosely dotted] (72.east) -- (74);
        \draw[loosely dotted] (52.east) -- (54);
        \draw[loosely dotted] (54.east) -- (56);
        \draw[loosely dotted] (34.east) -- (36);
        \draw[loosely dotted] (45.east) -- (47);
        \draw[loosely dotted] (47.east) -- (49);
        \draw[loosely dotted] (58.east) -- (510);
        \draw[loosely dotted] (38.east) -- (310);
        \draw[loosely dotted] (510.east) -- (512);
        \draw[loosely dotted] (310.east) -- (312);
        \draw[loosely dotted] (110.east) -- (112);
        \draw[loosely dotted] (211.east) -- (213);
        \draw[loosely dotted] (512.east) -- (514);
    \end{tikzpicture}
    \]
    
    \item Let $Q$ be as in Example \ref{ex:n-pre-admissible quivers}(d). The Aus\-lan\-der--Rei\-ten quiver of $\K Q/\J^2$ is
    \[
    \begin{tikzpicture}[scale=0.9, transform shape, baseline={(current bounding box.center)}]
        \tikzstyle{nct}=[circle, minimum width=0.6cm, draw, inner sep=0pt, text centered, scale=0.9]
        \tikzstyle{nout}=[circle, minimum width=6pt, draw=none, inner sep=0pt, scale=0.9]
    
    \node[nout] (1) at (0,2.1) {$\begin{smallmatrix}
    1
    \end{smallmatrix}$};
    \node[nout] (81) at (0.7,2.8) {$\begin{smallmatrix}
    8 \\ 1
    \end{smallmatrix}$};
    \node[nout] (8) at (1.4,2.1) {$\begin{smallmatrix}
    8
    \end{smallmatrix}$};
    \node[nout] (108) at (2.1,2.8) {$\begin{smallmatrix}
    10\;\;\;\;\; \\  8
    \end{smallmatrix}$};
    \node[nout] (78) at (2.1,1.4) {$\begin{smallmatrix}
    \;\;\;7 \\ 8
    \end{smallmatrix}$};
    \node[nout] (1078) at (2.8,2.1) {$\begin{smallmatrix}
    10 \;\; 7\; \\ 8
    \end{smallmatrix}$};
    
    \node[nout] (7) at (3.5, 2.8) {$\begin{smallmatrix}
    7
    \end{smallmatrix}$};
    \node[nout] (67) at (4.2,3.5) {$\begin{smallmatrix}
    6 \\ 7
    \end{smallmatrix}$};
    \node[nout] (11) at (4.9,4.2) {$\begin{smallmatrix}
    11
    \end{smallmatrix}$};
    \node[nout] (6) at (4.9,2.8) {$\begin{smallmatrix}
    6
    \end{smallmatrix}$};
    \node[nout] (5611) at (5.6,3.5) {$\begin{smallmatrix}
    5\\ 11 \;\; 6\;
    \end{smallmatrix}$};
    \node[nout] (56) at (6.3,4.2) {$\begin{smallmatrix}
    5 \\ \;\;\;6
    \end{smallmatrix}$};
    \node[nout] (511) at (6.3,2.8)
    {$\begin{smallmatrix}
    5 \\ 11\;\;\;\;
    \end{smallmatrix}$};
    \node[nout] (5) at (7,3.5) {$\begin{smallmatrix}
    5
    \end{smallmatrix}$};
    \node[nout] (45) at (7.7,4.2) {$\begin{smallmatrix}
    4 \\ 5
    \end{smallmatrix}$};
    \node[nout] (4) at (8.4, 3.5) {$\begin{smallmatrix}
    4
    \end{smallmatrix}$};
    \node[nout] (34) at (9.1,4.2) {$\begin{smallmatrix}
    3\;\;\; \\ 4
    \end{smallmatrix}$};
    \node[nout] (124) at (9.1,2.8) {$\begin{smallmatrix}
    \;\;\;\;12 \\ 4
    \end{smallmatrix}$};
    \node (3124) at (9.8,3.5) {$\begin{smallmatrix}
    \;3 \;\;12 \\ 4
    \end{smallmatrix}$};
    \node[nout] (12) at (10.5,4.2) {$\begin{smallmatrix}
    12
    \end{smallmatrix}$};
    \node[nout] (3) at (10.5,2.8) {$\begin{smallmatrix}
    3
    \end{smallmatrix}$};
    \node[nout] (1112) at (11.2,4.9) {$\begin{smallmatrix}
    11 \\ 12
    \end{smallmatrix}$};
    \node[nout] (23) at (11.2,2.1) {$\begin{smallmatrix}
    2 \\ 3
    \end{smallmatrix}$};
    \node[nout] (11b) at (11.9,4.2) {$\begin{smallmatrix}
    11
    \end{smallmatrix}$};
    \node[nout] (2b) at (11.9,2.8) {$\begin{smallmatrix}
    2.
    \end{smallmatrix}$};
    
    \node[nout] (10)  at (3.5,1.4) {$\begin{smallmatrix}
    10
    \end{smallmatrix}$};
    \node[nout] (910) at (4.2,0.7) {$\begin{smallmatrix}
    9 \\ 10
    \end{smallmatrix}$};
    \node[nout] (9) at (4.9,1.4) {$\begin{smallmatrix}
    9
    \end{smallmatrix}$};
    \node[nout] (2) at (4.9,0) {$\begin{smallmatrix}
    2
    \end{smallmatrix}$};
    \node[nout] (129) at (5.6,0.7) {$\begin{smallmatrix}
    1 \\ 9\;\;2
    \end{smallmatrix}$};
    \node[nout] (1 2) at (6.3,1.4) {$\begin{smallmatrix}
    1 \\ \;\;\;2
    \end{smallmatrix}$};
    \node[nout] (19) at (6.3,0) {$\begin{smallmatrix}
    1 \\ 9\;\;\;
    \end{smallmatrix}$};
    \node[nout] (1b) at (7,0.7) {$\begin{smallmatrix}
    1
    \end{smallmatrix}$};

    \draw[->] (1) to (81);
    \draw[->] (81) to (8);
    \draw[->] (8) to (108);
    \draw[->] (8) to (78);
    \draw[->] (108) to (1078);
    \draw[->] (78) to (1078);
    \draw[->] (1078) to (7);
    \draw[->] (1078) to (10);
    
    \draw[->] (7) to (67);
    \draw[->] (67) to (6);
    \draw[->] (11) to (5611);
    \draw[->] (6) to (5611);
    \draw[->] (5611) to (56);
    \draw[->] (5611) to (511);
    \draw[->] (56) to (5);
    \draw[->] (511) to (5);
    \draw[->] (5) to (45);
    \draw[->] (45) to (4);
    \draw[->] (4) to (34);
    \draw[->] (4) to (124);
    \draw[->] (34) to (3124);
    \draw[->] (124) to (3124);
    \draw[->] (3124) to (12);
    \draw[->] (3124) to (3);
    \draw[->] (12) to (1112);
    \draw[->] (3) to (23);
    \draw[->] (1112) to (11b);
    \draw[->] (23) to (2b);
    
    \draw[->] (10) to (910);
    \draw[->] (910) to (9);
    \draw[->] (9) to (129);
    \draw[->] (2) to (129);
    \draw[->] (129) to (1 2);
    \draw[->] (129) to (19);
    \draw[->] (1 2) to (1b);
    \draw[->] (19) to (1b);
    
    \draw[loosely dotted] (1.east) -- (8);
    \draw[loosely dotted] (8.east) -- (1078);
    \draw[loosely dotted] (108.east) -- (7);
    \draw[loosely dotted] (78.east) -- (10);
    
    \draw[loosely dotted] (7.east) -- (6);
    \draw[loosely dotted] (6.east) -- (511);
    \draw[loosely dotted] (11.east) -- (56);
    \draw[loosely dotted] (5611.east) -- (5);
    \draw[loosely dotted] (5.east) -- (4);
    \draw[loosely dotted] (4.east) -- (3124);
    \draw[loosely dotted] (34.east) -- (12);
    \draw[loosely dotted] (124.east) -- (3);
    \draw[loosely dotted] (12.east) -- (11b);
    \draw[loosely dotted] (3.east) -- (2b);
    
    \draw[loosely dotted] (10.east) -- (9);
    \draw[loosely dotted] (9.east) -- (1 2);
    \draw[loosely dotted] (2.east) -- (19);
    \draw[loosely dotted] (129.east) -- (1b);
    
    \end{tikzpicture}
    \]
\end{enumerate}

\end{example}

\subsection{Flow paths}

We have seen that if $\La=\K Q/\J^2$ and $\cC\subseteq \m\La$ is an $n$-cluster tilting subcategory, then $Q$ is $n$-pre-admissible. The opposite is not true in general. It turns out that there are additional properties that $Q$ must satisfy. To describe these properties we need to consider certain paths in $Q$.

\begin{definition}\label{def:flow path}
Let $Q$ be an $n$-pre-admissible quiver and let $k\geq 2$. A \emph{($k$-)flow path $\mathbf{v}$ in Q} is a path 
\begin{equation}\label{eq:flow path}
\mathbf{v} = 
    \begin{tikzcd} 
        v_1 \arrow[r] & v_2 \arrow[r] & \cdots \arrow[r] & v_{k-1} \arrow[r] & v_k,
    \end{tikzcd}
\end{equation}
such that $\delta(v_s)=(1,1)$ if and only if $1<s<k$.
\end{definition}

Notice that since $Q$ is $n$-pre-admissible, there are no multiple arrows between two vertices. Hence a flow path is defined uniquely by its vertices and we do not need to label arrows in a flow path. For a flow path $\mathbf{v}$ in $Q$ we use $v_i$ to denote its vertices as in (\ref{eq:flow path}). Moreover, in what follows we write ``$k$-flow path'' when the length $k$ of the flow path is important and ``flow path'' otherwise. Many of the results presented in this section have a dual version which although we omit for brevity, we sometimes use. We first study the case where there are no flow paths in $Q$.

\begin{lemma}\label{lem:the existence of flow paths}
Let $Q$ be an $n$-pre-admissible quiver. Then there exists a flow path in $Q$ if and only if $Q\neq A_1$ and $Q\neq \tilde{A}_m$ for some $m\geq 1$.
\end{lemma}

\begin{proof}
It is clear by the definition of a flow path that if there exists a flow path in $Q$, then $Q\neq A_1$ and $Q\neq \tilde{A}_m$. For the other direction, assume that $Q\neq A_1$ and $Q\neq\tilde{A}_m$ and we show that there exists a flow path in $Q$. Since $Q$ is connected and $Q\neq \tilde{A}_m$, there exists a vertex $v_1$ in $Q$ with degree $\delta(v_1)\neq (1,1)$. Since $Q\neq A_1$, we have $\delta(v_1)\neq(0,0)$. Since $Q$ is finite, any path starting or ending at $v_1$ eventually passes through a vertex $v_k$ with $\delta(v_k)\neq (1,1)$; let $\mathbf{v}$ be a minimal such path. Then $\mathbf{v}$ is a flow path by definition.
\end{proof}

\begin{proposition}\label{prop:the cases A_1 and A_m tilde}
Let $Q$ be an $n$-pre-admissible quiver and let $\La=\K Q/\J^2$.
\begin{enumerate}[label=(\alph*)]
    \item If $Q=A_1$, then $\cC$ is an $n$-cluster tilting subcategory of $\m\La$ if and only if $\cC=\mod\La=\add(\La)$.
    \item If $Q=\tilde{A}_m$ for some $m\geq 1$, then $\cC$ is an $n$-cluster tilting subcategory of $\m\La$ if and only if $n\divides m$ and $\cC=\add\left(\La\oplus \left(\bigoplus_{j= 0}^{\frac{m}{n}-1}\tau_n^{-j}(S)\right)\right)$ for some simple module $S\in\m\La$.
\end{enumerate}
\end{proposition}

\begin{proof}
\begin{enumerate}[label=(\alph*)]
    \item In this case $\La=\K$ and the result is clear.
    \item Follows from \cite[Theorem 5.1]{DI}.\qedhere
\end{enumerate}
\end{proof}

By Lemma \ref{lem:the existence of flow paths} we have that the only $n$-pre-admissible quivers that do not have flow paths are the quivers $A_1$ and $\tilde{A}_m$ for $m\geq 1$. Proposition \ref{prop:the cases A_1 and A_m tilde} classifies radical square zero bound quiver algebras with such quivers that admit $n$-cluster tilting subcategories. Hence it remains to study $n$-pre-admissible quivers that have flow paths. For the rest of this section we fix an $n$-pre-admissible quiver $Q$ such that $Q\neq A_1$ and $Q\neq \tilde{A}_m$ for any $m\geq 1$. It then follows that there exists a flow path in $Q$. We further set $\La\coloneqq \K Q/\J^2$. We start with some simple but important observations about flow paths.

\begin{lemma}\label{lem:same vertices in k-flow path}
Let $\mathbf{v}$ be a $k$-flow path in $Q$. Let $1\leq s\leq t\leq k$. 
\begin{enumerate}[label=(\alph*)]
    \item If $1<s$ and $t<k$, then $v_s=v_t$ if and only if $s=t$.
    \item If $s<t$ and $v_s=v_t$, then $s=1$ and $t=k$. In particular, in this case $v_1=v_k$. 
\end{enumerate}
\end{lemma}

\begin{proof}
\begin{enumerate}[label=(\alph*)]
    \item If $s=t$, then clearly $v_s=v_t$. Assume towards a contradiction that $v_s=v_t$ but $s<t$. Without loss of generality, we may assume that $s<t$ are minimal among $\{2,\ldots,k-1\}$ with these properties. By the definition of a $k$-flow path and since $\delta(v_s)=\delta(v_t)=(1,1)$, it follows that $v_{s-1}=v_{t-1}$. By minimality of $s$ and $t$ we conclude that $s-1=1$. Moreover, we have $1< s\leq t-1< t< k$ and so $\delta(v_{t-1})=(1,1)$. Then
    \[
    (1,1)\neq \delta(v_{1})=\delta(v_{s-1})=\delta(v_{t-1})=(1,1),
    \]
    which is a contradiction.
    
    \item Since $s<t$ and $v_s=v_t$, it follows from (a) that $s=1$ or $t=k$. In both cases we get that $\delta(v_s)=\delta(v_t)\neq (1,1)$. It follows from the definition of a $k$-flow path that $s=1$ and $t=k$. \qedhere
\end{enumerate}
\end{proof}

\begin{lemma}\label{lem:same flow paths}
Let $\mathbf{v}$ be a $k$-flow path in $Q$ and let $\mathbf{u}$ be a $k'$-flow path in $Q$. 
\begin{enumerate}[label=(\alph*)]
    \item Let $v_s$ be a vertex in $\mathbf{v}$ with $\delta(v_s)=(1,1)$ and assume that $v_s=u_t$ for some vertex $u_t$ in $\mathbf{u}$. Then $\mathbf{v}=\mathbf{u}$.
    
    \item Assume that $v_k=u_{k'}$ and that $v_{k-1}=u_{k'-1}$. Then $\mathbf{v}=\mathbf{u}$.
    
    \item Assume that $v_k=u_{k'}$ and that $\delta^{-}(v_k)=1$. Then $\mathbf{v}=\mathbf{u}$.
\end{enumerate}

\end{lemma}

\begin{proof}
\begin{enumerate}[label=(\alph*)]
    \item Since $\delta(v_s)=(1,1)$, it follows from the definition of a flow path that $1<s<k$. Since $v_s=u_t$ it follows that $\delta(u_t)=(1,1)$. Without loss of generality we may assume that $s\leq t$. By the definition of a $k$-flow path it follows that $v_{s-1}=u_{t-1}$. Continuing inductively we see that $v_{2}=u_{t-(s-2)}$ and this vertex has degree $(1,1)$. Hence the only arrow ending at $v_{2}=u_{t-(s-2)}$ is the arrow coming from $v_1$. Since $\delta(v_1)\neq (1,1)$ and $\delta(u_{t-(s-2)})=(1,1)$, and since there exists an arrow $v_1\longrightarrow u_{t-(s-2)}$, it follows that $u_1=v_1$ and $u_2=u_{t-(s-2)}$. By Lemma \ref{lem:same vertices in k-flow path}(a) it follows that $2=t-(s-2)$ and so $s=t$. A dual argument shows that $v_k=u_{k'}$. Since $v_1=u_1$, $v_k=u_k$ and $v_s=u_s$ for some $s$ with $1<s<k$, it readily follows that $\mathbf{v}=\mathbf{u}$.
    
    \item Since $\mathbf{v}$ and $\mathbf{u}$ are flow paths, there exist arrows $\alpha: v_{k-1}\longrightarrow v_k$ and $\beta:u_{k'-1}\longrightarrow u_{k'}$. Since $v_{k-1}=u_{k'-1}$ and $v_k=u_{k'}$, and since there exist no multiple arrows between two vertices of $Q$, it follows that $\alpha=\beta$. If $\delta(v_{k-1})\neq (1,1)$, then $\mathbf{v}=v_{k-1}\longrightarrow v_k = u_{k'-1} \longrightarrow u_{k'}=\mathbf{u}$, as required. Otherwise, if $\delta(v_{k-1})=(1,1)$, then the result follows from (a).
    
    \item Since $\delta^{-}(v_k)=\delta^{-}(u_{k'})=1$, there exists a unique arrow ending at $v_k=u_{k'}$. Since there exist arrows $v_{k-1}\longrightarrow v_k$ and $u_{k'-1}\longrightarrow u_{k'}$, we conclude that $v_{k-1}=u_{k'-1}$. The result follows from (b). \qedhere
\end{enumerate}
\end{proof}

\begin{corollary}\label{cor:unique flow path}
Let $\alpha:w\longrightarrow v$ be an arrow in $Q$.
\begin{enumerate}[label=(\alph*)]
    \item If $\delta(v)=(1,1)$, then there exists a unique flow path $\mathbf{v}$ in $Q$ through $v$. 
    \item If $\delta(v)\neq (1,1)$, then there exists a unique flow path $\mathbf{v}$ in $Q$ ending at $v$ such that $\alpha$ is the last arrow of $\mathbf{v}$.
\end{enumerate}
In both cases we have that $v=v_j$ for some $j>1$.
\end{corollary}

\begin{proof}
The existence of the flow path is clear since $Q\neq A_1$ and $Q\neq \tilde{A}_m$. The uniqueness follows from Lemma \ref{lem:same flow paths}.
\end{proof}

If $\mathbf{v}$ is a $k$-flow path in $Q$, then the arrows ending and starting at the vertices $v_1$ and $v_k$ play an important role in our investigation. Hence we also label the following vertices
\begin{equation}\label{eq:extended flow path}
\begin{tikzcd} 
    v^{-2} \arrow[rd, dotted] & {} & v^{-1} & {} & v^{+1} \arrow[rd, dotted] & {} & v^{+2} \\
    {} & v_1 \arrow[r] \arrow[ru, dotted] & v_2 \arrow[r] & \cdots \arrow[r] & v_{k-1} \arrow[r] & v_k \arrow[ru, dotted] \arrow[rd, dotted] & {} \\
    v^{-3} \arrow[ru, dotted] & {} & {} & {} & {} & {} & v^{+3},
\end{tikzcd}
\end{equation}
where a dotted arrow means that such an arrow may or may not exist. When $\delta^{-}(v_1)=1$, we assume that the arrow $v^{-2}\longrightarrow v_1$ is the one that exists and when $\delta^{+}(v_k)=1$ we assume that the arrow $v_k\longrightarrow v^{+2}$ is the one that exists. Notice that by Definition \ref{def:n-pre-admissible}(ii) we have that $v^{-1}\neq v_2$ and $v^{+1}\neq v_{k-1}$, if the vertices $v^{-1}$ and $v^{+1}$ exist.

We also set
\[
I(\mathbf{v}) \coloneqq \begin{cases}
    I(v_1), &\mbox{if $\delta^{+}(v_1)=1$,} \\ 
    I(v^{-1}), &\mbox{if $\delta^{+}(v_1)=2$,}
\end{cases}
\text{\;\; and  \;\;}
P(\mathbf{v}) \coloneqq \begin{cases}
    P(v_k), &\mbox{if $\delta^{-}(v_k)=1$,} \\ 
    P(v^{+1}), &\mbox{if $\delta^{-}(v_k)=2$.}
\end{cases}
\]
With this notation, we have the following technical results.

\begin{lemma}\label{lem:same P(v) implies same path}
Let $\mathbf{v}$ be a $k$-flow path in $Q$ and let $\mathbf{u}$ be a $k'$-flow path in $Q$. Then $P(\mathbf{v})$ is not injective and $P(\mathbf{v})\isom P(\mathbf{u})$ if and only if $\mathbf{v}=\mathbf{u}$.
\end{lemma}

\begin{proof}
That $P(\mathbf{v})$ is not injective follows immediately by the definition of $P(\mathbf{v})$ and since $\delta(v_k)\in\{(1,0),(1,2),(2,1),(2,2)\}$. That $\mathbf{v}=\mathbf{u}$ implies $P(\mathbf{v})\isom P(\mathbf{u})$ is clear. Now assume that $P(\mathbf{v})\isom P(\mathbf{u})$ and we show that $\mathbf{v}=\mathbf{u}$. We first claim that $\delta^{-}(v_k)=\delta^{-}(u_{k'})$. Indeed, assume towards a contradiction that $\delta^{-}(v_k)\neq \delta^{-}(u_{k'})$. Without loss of generality we may assume that $\delta^{-}(v_k)=1$ and $\delta^{-}(u_{k'})=2$. Then $P(\mathbf{v})=P(v_k)$ and $P(\mathbf{u})=P(u^{+1})$. Hence $v_k=u^{+1}$. By Definition \ref{def:n-pre-admissible}(iii), it follows that $\delta^{+}(u^{+1})=1$. Therefore $\delta(v_k)=(\delta^{-}(v_k),\delta^{+}(v_k))=(1,\delta^{+}(u^{+1}))=(1,1)$, which contradicts the definition of a $k$-flow path.

Hence we have shown that $\delta^{-}(v_k)=\delta^{-}(u_{k'})$. Next we consider the cases $\delta^{-}(v_k)=1$ and $\delta^{-}(v_k)=2$ separately.
 
Case $\delta^{-}(v_k)=1$. In this case $\delta^{-}(u_{k'})=1$ and so $P(v_k)\isom P(u_{k'})$. It follows that $v_k=u_{k'}$. Therefore we have that $\mathbf{v}=\mathbf{u}$ by Lemma \ref{lem:same flow paths}(c).

Case $\delta^{-}(v_k)=2$. In this case $\delta^{-}(u_{k'})=2$ and so $P(v^{+1})\isom P(u^{+1})$. It follows that $v^{+1}=u^{+1}$. Since $\delta^{+}(v^{+1})=1$ and there exist an arrow $v^{+1}\longrightarrow v_k$ and an arrow $v^{+1}=u^{+1}\longrightarrow u_{k'}$, it follows that $v_k=u_{k'}$. Since $v^{+1}=u^{+1}\neq u_{k'-1}$, it follows that $u_{k'-1}=v_{k-1}$. Therefore we have that $\mathbf{v}=\mathbf{u}$ by Lemma \ref{lem:same flow paths}(b).
\end{proof}

\begin{lemma}\label{lem:indecomposable projectives noninjectives come from flow paths, are injective or have degree (2,2)}
Let $v\in Q_0$ be a vertex. Then exactly one of the following three conditions hold:
\begin{enumerate}[label=(\roman*)]
    \item $P(v)$ is injective.
    \item $\delta(v)=(2,2)$.
    \item $P(v)=P(\mathbf{v})$ for some flow path $\mathbf{v}$.
\end{enumerate}
\end{lemma}

\begin{proof}
Notice that conditions (i) and (iii) cannot hold simultaneously since by Lemma \ref{lem:same P(v) implies same path} we have that $P(\mathbf{v})$ is not injective. Moreover, by the definition of $P(\mathbf{v})$, conditions (ii) and (iii) also cannot hold simultaneously. It is also clear that conditions (i) and (ii) cannot hold simultaneously, since if $\delta(v)=(2,2)$, then $P(v)$ does not have simple socle. Hence it is enough to show that one of the conditions (i),(ii) or (iii) holds. We consider the cases $\delta^{+}(v)=0$, $\delta^{+}(v)=1$ and $\delta^{+}(v)=2$ separately.

Case $\delta^{+}(v)=0$. In this case $\delta(v)=(1,0)$ and by Corollary \ref{cor:unique flow path}(b) there exists a unique flow path $\mathbf{v}$ ending at $v$. It follows from the definition of $P(\mathbf{v})$ that $P(\mathbf{v})=P(v)$ and so condition (iii) holds.

Case $\delta^{+}(v)=1$. Let $\alpha:v\longrightarrow u$ be the unique arrow starting at $v$. We consider the subcases $\delta^{-}(u)=1$ and $\delta^{-}(u)=2$ separately.
\begin{itemize}
    \item Subcase $\delta^{-}(u)=1$. In this case $P(v)=I(u)$ is injective and so condition (i) holds.
    
    \item Subcase $\delta^{-}(u)=2$. Let $\beta:w\longrightarrow u$ be the other arrow ending at $u$. By Corollary \ref{cor:unique flow path}(b) and since $\delta(u)\neq(1,1)$, it follows that there exists a unique flow path $\mathbf{v}$ such that the last arrow of $\mathbf{v}$ is $\beta$. It follows from the definition of $P(\mathbf{v})$ that $P(\mathbf{v})=P(v)$ and so condition (iii) holds.
\end{itemize}

Case $\delta^{+}(v)=2$. We consider the subcases $\delta(v)=(1,2)$ and $\delta(v)=(2,2)$ separately.
\begin{itemize}
    \item Subcase $\delta(v)=(1,2)$. By Corollary \ref{cor:unique flow path}(b) there exists a unique flow path $\mathbf{v}$ ending at $v$. It follows from the definition of $P(\mathbf{v})$ that $P(\mathbf{v})=P(v)$ and so condition (iii) holds.
    
    \item Subcase $\delta(v)=(2,2)$. In this case condition (ii) holds. \qedhere
\end{itemize}
\end{proof}

Let $v\in Q_0$ be a vertex. If there exists an $n$-cluster titling subcategory $\cC\subseteq\m\La$, then we have that $P(v)\in\cC$. By Proposition \ref{prop:n-ct is closed under n-AR translations and uniqueness of n-ct}(a) we then have that $\tau_n^{-j}(P(v))\in\cC$ for all $j\geq 0$. By Lemma \ref{lem:indecomposable projectives noninjectives come from flow paths, are injective or have degree (2,2)} there are three different cases for $P(v)$. If $P(v)$ belongs to the first case, that is if $P(v)$ is injective, then $\tau_n^{-j}(P(v))=0$ for $j\geq 1$. Our aim now is to compute $\tau_n^{-j}(P(v))$ for the two remaining cases. To this end we need the following lemma.

\begin{lemma}\label{lem:cosyzygy and tau-inverse of simples}
Let $v\in Q_0$ be a vertex. 
\begin{enumerate}[label=(\alph*)]
    \item If $\delta^{-}(v)=0$, then $\Omega^{-}(S(v))=\tau^{-}(S(v))=0$.
    \item If $\delta^{-}(v)=1$, let $w\longrightarrow v$ be the unique arrow ending at $v$. Then $\Omega^{-}(S(v))\isom S(w)$ and $\tau^{-}(S(v))\isom \coker(S(v)\hookrightarrow P(w))$.
    \item If $\delta^{-}(v)=2$, let $w_1\longrightarrow v$ and $w_2\longrightarrow v$ be the two arrows ending at $v$. Then $\Omega^{-}(S(v))\isom S(w_1)\oplus S(w_2)$ and $\tau^{-}(S(v))\isom I(v)$.
\end{enumerate}
\end{lemma}

\begin{proof}
By Theorem \ref{thrm:representation-finite string algebra} the algebra $\La$ is a string algebra. The Aus\-lan\-der--Rei\-ten translations for modules over string algebras are computed in \cite{ButRin}. We include here a simple proof in this special case. 
\begin{enumerate}[label=(\alph*)]
    \item If $\delta^{-}(v)=0$, then $S(v)=I(v)$ is injective and so $\Omega^{-}(S(v))=\tau^{-}(S(v))=0$.
    \item Since $\La$ is a radical square zero algebra and $\delta^{-}(v)=1$, we have that $I(v)=\begin{smallmatrix}
     w \\ v \end{smallmatrix}$. Hence there exists a minimal injective presentation of $S(v)$ of the form
    \[
    \begin{tikzcd}[column sep=small, row sep=small]
        0 \arrow[rr] &&
        S(v) \arrow[rr, "i_0"] && 
        I(v) \arrow[rr, "i_1"] \arrow[rd, two heads] && 
        I(w), \\
        {} && {} && {} & S(w) \arrow[ru, hook] && {}
    \end{tikzcd}
    \]
    from which it follows that $\Omega^{-}(S(v))\isom S(w)$. Furthermore, by applying the inverse Nakayama functor $\nu^{-}$ to the above presentation we obtain an exact sequence
    \[
    \begin{tikzcd}[column sep=small, row sep=small]
        0 \arrow[rr] && 
        \nu^{-}(S(v)) \arrow[rr, "\nu^{-}(i_0)"] &&
        P(v) \arrow[rr, "\nu^{-}(i_1)"] \arrow[rd, two heads] && 
        P(w) \arrow[rr]  && 
        \tau^{-}(S(v)), \\
        {} && {} && {} & S(v) \arrow[ru, hook] && {} && {}
    \end{tikzcd}
    \] 
    from which it follows that $\tau^{-}(S(v))\isom \coker(S(v)\hookrightarrow P(w))$.
    
    \item Since $\La$ is a radical square zero algebra and $\delta^{-}(v)=2$, we have that $I(v)=\begin{smallmatrix}
     w_1\;\;w_2 \\ v \end{smallmatrix}$. Hence there exists a minimal injective presentation of $S(v)$ of the form 
     \[
    \begin{tikzcd}[column sep=small, row sep=small]
        0 \arrow[rr] &&
        S(v) \arrow[rr, "i_0"] && 
        I(v) \arrow[rr, "i_1"] \arrow[rd, two heads] && 
        I(w_1)\oplus I(w_2), \\
        {} && {} && {} & S(w_1)\oplus S(w_2) \arrow[ru, hook] && {}
    \end{tikzcd}
    \]
    from which it follows that $\Omega^{-}(S(v))\isom S(w_1)\oplus S(w_2)$. By applying the inverse Nakayama functor $\nu^{-}$ to the above presentation we obtain an exact sequence
    \[
    \begin{tikzcd}[column sep=small, row sep=small]
        0 \arrow[rr] && 
        \nu^{-}(S(v)) \arrow[rr, "\nu^{-}(i_0)"] &&
        P(v) \arrow[rr, "\nu^{-}(i_1)"] \arrow[rd, two heads] && 
        P(w_1)\oplus P(w_2) \arrow[rr]  && 
        \tau^{-}(S(v)). \\
        {} && {} && {} & S(v) \arrow[ru, hook] && {} && {}
    \end{tikzcd}
    \]
    By Definition \ref{def:n-pre-admissible}(iii) we have that $P(w_1)=\begin{smallmatrix}
    w_1 \\ v
    \end{smallmatrix}$ and $P(w_2)=\begin{smallmatrix}
    w_2 \\ v
    \end{smallmatrix}$. Then $\coker(S(v)\hookrightarrow P(w_1)\oplus P(w_2))\isom I(v)$ and the result follows. \qedhere
\end{enumerate}
\end{proof}

We can now compute $\tau_n^{-j}(P(v))$ in the second case of Lemma \ref{lem:indecomposable projectives noninjectives come from flow paths, are injective or have degree (2,2)}, that is when $\delta(v)=(2,2)$. Notice that in this case we have $n=2$ by Definition \ref{def:n-pre-admissible}(i).

\begin{corollary}\label{cor:taun-inverse of P(v) for degree (2,2)}
Let $v\in Q_0$ be a vertex with $\delta(v)=(2,2)$. Then $\tau_2^{-}(P(v))\isom I(v)$.
\end{corollary}

\begin{proof}
Let $v\longrightarrow u_1$ and $v\longrightarrow u_2$ be the arrows starting at $v$. By Definition \ref{def:n-pre-admissible}(iii) we have that $\delta^{-}(u_1)=\delta^{-}(u_2)=1$. It follows that $\Omega^{-}(P(v))\isom S(v)$. By Lemma \ref{lem:cosyzygy and tau-inverse of simples}(c) we have that $\tau^{-}(S(v))\isom I(v)$. Hence 
\[
\tau_2^{-}(P(v)) = \tau^{-}\Omega^{-}(P(v)) \isom \tau^{-}(S(v))\isom I(v),
\]
as required.
\end{proof}

Before continuing with the computation of $\tau_n^{-j}(P(v))$ in the last case, that is when $P(v)=P(\mathbf{v})$ for a flow path $\mathbf{v}$ in $Q$, let us introduce one more piece of notation.

\begin{definition}\label{def:q of flow path}
Let $\mathbf{v}=v_1\longrightarrow v_2 \longrightarrow \cdots \longrightarrow v_k$ be a $k$-flow path. We define
\[q_1= q_1(\mathbf{v}) \coloneqq \begin{cases} 1, &\mbox{if $\delta(v_1)=(2,1)$,} \\ 0, &\mbox{if $\delta(v_1)\neq (2,1)$,}\end{cases} \text{ and }
q_k=q_k(\mathbf{v})\coloneqq \begin{cases} 1, &\mbox{if $\delta(v_k)=(1,2)$,} \\ 0, &\mbox{if $\delta(v_k)\neq (1,2)$.}\end{cases}\]
We also define 
\[
q(\mathbf{v}) \coloneqq -1+q_1+q_k = \begin{cases} 1, &\mbox{if $\delta(v_1)=(2,1)$ and $\delta(v_k)=(1,2)$,} \\ 0, &\mbox{if either $\delta(v_1)=(2,1)$ or $\delta(v_k)=(1,2)$,} \\ -1, &\mbox{if $\delta(v_1)\neq (2,1)$ and $\delta(v_k)\neq (1,2)$.}\end{cases}
\]
\end{definition}

With this definition we can write some of the next results in a more compact way. First we have the following statement.

\begin{lemma}\label{lem:vertex between 2 and k-1+q_k has incoming degree 1}
Let $\mathbf{v}$ be a $k$-flow path in $Q$. Let $s\in \ZZ$ and assume that $2\leq s\leq k-1+q_k$. Then $\delta^{-}(v_s)=1$.
\end{lemma}

\begin{proof}
We have $s\leq k-1+q_k\leq k$. We consider the cases $s\leq k-1$ and $s=k$ separately. If $s\leq k-1$, then $\delta(v_s)=(1,1)$ by the definition of flow paths and so the result holds. If $s=k$, then $k-1+q_k=k$ and so $q_k=1$. Then by the definition of $q_k$ we have $\delta(v_s)=\delta(v_k)=(1,2)$, and so the result holds again.
\end{proof}

With this we are ready to make the following computations.

\begin{lemma}\label{lem:n-th cosyzygy and taun-inverse of simples in middle of k-flow paths}
Let $\mathbf{v}$ be a $k$-flow path in $Q$. Let $s,x\in \ZZ_{\geq 0}$ and assume that $1\leq s\leq k-1+q_k$. 
\begin{enumerate}[label=(\alph*)]
    \item If $s-x\geq 1$, then $\Omega^{-x}(S(v_s))\isom S(v_{s-x})$.
    \item If $1\leq x \leq s-1+q_1$, then $\tau_x^{-}(S(v_s)) \isom \begin{cases} S(v_{s-x}), &\mbox{if $1\leq x<s-1+q_1$,} \\ I(\mathbf{v}), &\mbox{if $x=s-1+q_1$.} \end{cases}$
\end{enumerate}
\end{lemma}

\begin{proof}
\begin{enumerate}[label=(\alph*)]
    \item We use induction on $x$. If $x=0$, then the result holds trivially. Assume now that the result holds for $x-1\geq 0$ and we show that it holds for $x$. Since $s-x\geq 1$, we have that $s-(x-1)\geq 1$. Hence by induction hypothesis we have that $\Omega^{-(x-1)}(S(v_s))\isom S(v_{s-(x-1)})$. Then
    \[
    2 =1+1\leq (s-x)+1 = s-(x-1) \leq s \leq k-1+q_k, 
    \]
    and so $\delta^{-}(v_{s-(x-1)})=1$ by Lemma \ref{lem:vertex between 2 and k-1+q_k has incoming degree 1}. Then by the definition of flow paths and Lemma \ref{lem:cosyzygy and tau-inverse of simples}(b) applied on $v_{s-(x-1)}$ it follows that 
    \[
    \Omega^{-x}(S(v_s)) \isom \Omega^{-}(S(v_{s-(x-1)})) \isom S(v_{s-x}),
    \]
    as required.
    
    \item Since $x \leq s-1+q_1$, we have that $s-(x-1)\geq 2-q_1\geq 1$. Therefore, by (a) we have that
    \[
    \tau_x^{-}(S(v_s)) = \tau^{-}\Omega^{-(x-1)}(S(v_s))\isom \tau^{-}(S(v_{s-(x-1)})).
    \]
    Hence it is enough to show that 
    \[
    \tau^{-}(S(v_{s-(x-1)}))\isom \begin{cases} S(v_{s-x}), &\mbox{if $1\leq x<s-1+q_1$,} \\ I(\mathbf{v}), &\mbox{if $x=s-1+q_1$.} \end{cases}
    \]
    We consider the cases $1\leq x < s-1+q_1$ and $x=s-1+q_1$ separately.
    
    Case $1\leq x <s-1+q_1$. In this case we want to show that $\tau^{-}(S(v_{s-(x-1)}))\isom S(v_{s-x})$. We have 
    \begin{equation}\label{eq:where does the vertex lie}
        2-q_1 < s-(x-1) \leq s \leq k-1+q_k\leq k.
    \end{equation}
    Hence $2\leq s-(x-1)\leq k-1+q_k$ and so by Lemma \ref{lem:vertex between 2 and k-1+q_k has incoming degree 1} we have that $\delta^{-}(v_{s-(x-1)})=1$. It follows from Lemma \ref{lem:cosyzygy and tau-inverse of simples}(b) that it is enough to show that $\delta^{+}(v_{s-x})=1$. We consider the subcases $q_1=0$ and $q_1=1$ separately. 
    \begin{itemize}
        \item Subcase $q_1=0$. Then by (\ref{eq:where does the vertex lie}) we conclude that $2\leq s-x\leq k-1$ and so $\delta^{+}(v_{s-x})=1$.
        \item Subcase $q_1=1$. Then by (\ref{eq:where does the vertex lie}) we conclude that $1\leq s-x\leq k-1$. Since in this case we have $\delta(v_1)=(2,1)$, it follows that $\delta^{+}(v_{s-x})=1$.
    \end{itemize}

    Case $x=s-1+q_1$. In this case we have $s-(x-1)=2-q_1$ and we want to show that $\tau^{-}(S(v_{2-q_1}))\isom I(\mathbf{v})$. We consider the cases $q_1=0$ and $q_1=1$ separately.
    \begin{itemize}
        \item Subcase $q_1=0$. Then the result follows immediately by Lemma \ref{lem:cosyzygy and tau-inverse of simples}(b) and by considering the possibilities $\delta(v_1)=(0,1)$, $\delta(v_1)=(1,2)$ and $\delta(v_1)=(2,2)$ separately.
        \item Subcase $q_1=1$. Then $\delta(v_1)=(2,1)$ and by Lemma \ref{lem:cosyzygy and tau-inverse of simples}(c) we have $\tau^{-}(S(v_1))\isom I(v_1) = I(\mathbf{v})$. \qedhere
    \end{itemize}
\end{enumerate}
\end{proof}

\begin{lemma}\label{lem:n-th cosyzygy and taun-inverse of P(v)}
Let $\mathbf{v}$ be a $k$-flow path in $Q$. Let $x\in\ZZ_{\geq 1}$.
\begin{enumerate}[label=(\alph*)]
    \item If $k-x+q_k\geq 1$, then $\Omega^{-x}(P(\mathbf{v}))\isom S(v_{k-x+q_k})$.
    
    \item If $1\leq x \leq k+q(\mathbf{v})$, then $\tau_x^{-}(P(\mathbf{v}))\isom \begin{cases} S(v_{k-x+q_k}), &\mbox{if $1\leq x<k+q(\mathbf{v})$,} \\ I(\mathbf{v}), &\mbox{if $x=k+q(\mathbf{v})$.} \end{cases}$
\end{enumerate}
\end{lemma}

\begin{proof}
\begin{enumerate}[label=(\alph*)]
    \item If $x=1$, then the result follows immediately by considering the cases $\delta(v_k)=(1,0)$, $\delta(v_k)=(1,2)$, $\delta(v_k)=(2,1)$ and $\delta(v_k)=(2,2)$ separately (recall that if $\delta(v_k)=(1,2)$, then $\delta^{-}(v^{+2})=\delta^{-}(v^{+3})=1$ by Definition \ref{def:n-pre-admissible}(iii)). For $x\geq 2$ notice that $1\leq k-x+q_k$ implies that $k-1+q_k-(x-1)\geq 1$. Hence we can apply Lemma \ref{lem:n-th cosyzygy and taun-inverse of simples in middle of k-flow paths}(a) to obtain
    \[
    \Omega^{-x}(P(\mathbf{v}))=\Omega^{-(x-1)}\Omega^{-}(P(\mathbf{v}))\isom \Omega^{-(x-1)}(S(v_{k-1+q_k})) \isom S(v_{k-x+q_k}),
    \]
    as required.
    
    \item We first show the result for $x=1$. We consider the cases $1=x < k+q(\mathbf{v})$ and $1=x=k+q(\mathbf{v})$ separately.
    
    Case $1=x<k+q(\mathbf{v})$. In this case we want to show that $\tau^{-}(P(\mathbf{v}))\isom S(v_{k-1+q_k})$. We consider the subcases $\delta(v_k)=(1,0)$, $\delta(v_k)=(1,2)$ and $\delta(v_k)\in\{(2,1),(2,2)\}$ separately.
    \begin{itemize}
        \item Subcase $\delta(v_k)=(1,0)$. In this case $q_k=0$ and  $P(\mathbf{v})=S(v_k)$ and so we want to show that $\tau^{-}(S(v_k))\isom S(v_{k-1})$. We claim that $\delta^{+}(v_{k-1})=1$. Indeed, assume towards a contradiction that $\delta^{+}(v_{k-1})=2$. Then $v_{k-1}=v_1$ and so $k=2$ and $q_1=0$. Hence $1=x<2+q(\mathbf{v})=2-1=1$, which is a contradiction. Hence by Lemma \ref{lem:cosyzygy and tau-inverse of simples}(b) and since $\delta^{+}(v_{k-1})=1$, it follows that $\tau^{-}(S(v_k))\isom S(v_{k-1})$.
        
        \item Subcase $\delta(v_k)=(1,2)$. In this case $q_k=1$ and $P(\mathbf{v})=P(v_{k})$ and so we want to show that $\tau^{-}(P(v_k))\isom S(v_{k})$. By the dual of Lemma \ref{lem:cosyzygy and tau-inverse of simples}(c) we have that $\tau(S(v_{k}))\isom P(v_k)$. By applying $\tau^{-}$ we obtain $\tau^{-}(P(v_k))\isom\tau^{-}\tau(S(v_{k}))\isom S(v_k)$.
        
        \item Subcase $\delta(v_k)\in\{(2,1),(2,2)\}$. In this case $q_k=0$ and $P(\mathbf{v})=P(v^{+1})$ and so we want to show that $\tau^{-}(P(v^{+1}))\isom S(v_{k-1})$. By Definition \ref{def:n-pre-admissible}(iii) we have that $\delta^{+}(v_{k-1})=1$. By the dual of Lemma \ref{lem:cosyzygy and tau-inverse of simples}(b) we then have that $\tau(S(v_{k-1}))\isom P(v^{+1})$. By applying $\tau^{-}$ we obtain $\tau^{-}(P(v^{+1}))\isom \tau^{-}\tau(S(v_{k-1}))\isom S(v_{k-1})$. 
    \end{itemize}
    
    Case $1=x=k+q(\mathbf{v})$. In this case we have that $k=2$ and $q(\mathbf{v})=-1$ and we want to show that $\tau^{-}(P(\mathbf{v}))\isom I(\mathbf{v})$. Since $q(\mathbf{v})=-1$, we have $\delta(v_1)\neq (2,1)$ and $\delta(v_2)\neq (1,2)$. We consider the subcases $\delta(v_2)=(1,0)$ and $\delta(v_2)\in\{(2,1),(2,2)\}$ separately.
    \begin{itemize}
        \item Subcase $\delta(v_2)=(1,0)$. In this case we have $P(\mathbf{v})=S(v_2)$ and so we want to show that $\tau^{-}(S(v_2))\isom I(\mathbf{v})$. If $\delta^{+}(v_1)=1$, since $\delta(v_1)\neq (2,1)$ and since by the definition of a flow path we have $\delta(v_1)\neq (1,1)$, we conclude that $\delta(v_1)=(0,1)$. Hence by Lemma \ref{lem:cosyzygy and tau-inverse of simples}(b) we have $\tau^{-}(S(v_2))\isom S(v_1)= I(\mathbf{v})$, where the last equality follows from the definition of $I(\mathbf{v})$. If $\delta^{+}(v_1)=2$, then by Lemma \ref{lem:cosyzygy and tau-inverse of simples}(b) and Definition \ref{def:n-pre-admissible}(iii) we have $\tau^{-}(S(v_2))\isom I(v^{-1}) = I(\mathbf{v})$, where the last equality again follows from the definition of $I(\mathbf{v})$.
        
        \item Subcase $\delta(v_2)\in\{(2,1),(2,2)\}$. In this case we have $P(\mathbf{v})= P(v^{+1})$ and so we want to show that $\tau^{-}(P(v^{+1}))\isom I(\mathbf{v})$. Since $k=2$ and $\delta^{-}(v_2)=2$, by Definition \ref{def:n-pre-admissible}(iii) we have that $\delta^{+}(v_1)=1$. Since $\delta(v_1)\neq (2,1)$ and since by the definition of a flow path we have $\delta(v_1)\neq (1,1)$, we conclude that $\delta(v_1)=(0,1)$. It follows that $I(\mathbf{v})=S(v_1)$. By the dual of Lemma \ref{lem:cosyzygy and tau-inverse of simples}(b) we then have that $\tau(S(v_1))\isom P(v^{+1})$. By applying $\tau^{-}$ we obtain $\tau^{-}(P(v^{+1}))\isom \tau^{-}\tau(S(v_1))\isom S(v_1)=I(\mathbf{v})$. 
    \end{itemize}
    
    Now let $x\geq 2$. Then $2\leq k+q(\mathbf{v})$ gives $k-1+q_k\geq 1$. Hence by (a) we have that
    \[
    \tau_x^{-}(P(\mathbf{v})) = \tau_{x-1}^{-}\Omega^{-}(P(\mathbf{v})) \isom \tau_{x-1}^{-}(S(v_{k-1+q_k})).
    \]
    Moreover, since $1\leq k-1+q_k$ and 
    \[
    (k-1+q_k)-1+q_1=k+q(\mathbf{v})-1\geq x-1\geq 1,
    \]
    we can apply Lemma \ref{lem:n-th cosyzygy and taun-inverse of simples in middle of k-flow paths}(b) to obtain
    \[
    \tau_{x-1}^{-}(S(v_{k-1+q_k})) \isom \begin{cases} S(v_{k-1+q_k-(x-1)}), &\mbox{if $1 \leq x-1 < (k-1+q_k)-1+q_1$,} \\ I(\mathbf{v}), &\mbox{if $x-1=(k-1+q_k)-1+q_1$.} \end{cases}
    \]
    After simplifying the above expression, we get
    \[
    \tau_x^{-}(P(\mathbf{v})) \isom \tau_{x-1}^{-}(S(v_{k-1+q_k})) \isom \begin{cases} S(v_{k-x+q_k}), &\mbox{if $2\leq x <k+q(\mathbf{v})$,} \\ I(\mathbf{v}), &\mbox{if $x=k+q(\mathbf{v})$,} \end{cases}
    \]
    which proves the case $x\geq 2$. \qedhere
\end{enumerate}
\end{proof}

With the above computation we can show the following important results about flow paths in $Q$.

\begin{proposition}\label{prop:n divides k+q(v)}
Let $\mathbf{v}$ be a $k$-flow path in $Q$ and assume that $\cC\subseteq \m\La$ is an $n$-cluster tilting subcategory. Then $n\divides (k+q(\mathbf{v}))$.
\end{proposition}

\begin{proof}
We write $k+q(\mathbf{v})=pn+r$ where $p\in \ZZ_{\geq 0}$ and $0\leq r \leq n-1$. We first claim that $p\geq 1$. Indeed, assume towards a contradiction that $p=0$. Then $1\leq k+q(\mathbf{v})=r$. Hence by Lemma \ref{lem:n-th cosyzygy and taun-inverse of P(v)}(b) we have that $\tau_{r}^{-}(P(\mathbf{v}))\isom I(\mathbf{v})$. By Lemma \ref{lem:taun-inverse gives nonzero ext} we obtain $\Ext_{\La}^{r}(I(\mathbf{v}),P(\mathbf{v}))\neq 0$. But this contradicts the fact that $\cC$ is an $n$-cluster tilting subcategory, since $I(\mathbf{v}),P(\mathbf{v})\in\cC$ and $1\leq r\leq n-1$.

Hence $p\geq 1$ and it remains to show that $r=0$. Assume towards a contradiction that $r\geq 1$. Then 
\[
1\leq n \leq pn = k+q(\mathbf{v})-r < k+q(\mathbf{v}).
\]
Hence we can apply Lemma \ref{lem:n-th cosyzygy and taun-inverse of P(v)}(b) to obtain that $\tau_n^{-}(P(\mathbf{v}))\isom S(v_{k-n+q_k})$. Then we can apply Lemma \ref{lem:n-th cosyzygy and taun-inverse of simples in middle of k-flow paths}(b) repeatedly  $p-1$ more times to obtain
\[
\tau_n^{-p}(P(\mathbf{v}))\isom \tau_n^{-(p-1)}(S(v_{k-n+q_k})) \isom \tau_n^{-(p-2)}(S(v_{k-2n+q_k}))\isom\cdots \isom S(v_{k-pn+q_k}) = S(v_{r+1-q_1}).
\]
By Proposition \ref{prop:n-ct is closed under n-AR translations and uniqueness of n-ct}(a) and since $P(\mathbf{v})\in\cC$, it follows that $S(v_{r+1-q_1})\in\cC$. By Lemma \ref{lem:n-th cosyzygy and taun-inverse of simples in middle of k-flow paths}(b) we have $\tau_{r}^{-}(S(v_{r+1-q_1}))\isom I(\mathbf{v})$. By Lemma \ref{lem:taun-inverse gives nonzero ext} we obtain $\Ext_{\La}^{r}(I(\mathbf{v}),S(v_{r+1-q_1}))\neq 0$. But this contradicts the fact that $\cC$ is an $n$-cluster tilting subcategory, since $I(\mathbf{v}),S(v_{r+1-q_1})\in\cC$ and $1\leq r\leq n-1$.
\end{proof}

\begin{corollary}\label{cor:j-th taun-inverse of P(v)}
Let $\mathbf{v}$ be a $k$-flow path in $Q$. Assume that $k+q(\mathbf{v})=pn$ for some $p\geq 1$ and let $j\in\ZZ$ with $0\leq j\leq p$. Then
\begin{equation}\label{eq:j-th taun-inverse of P(v)}
\tau_n^{-j}(P(\mathbf{v})) \isom \begin{cases} P(\mathbf{v}), &\mbox{if $j=0$,} \\ S(v_{k-jn+q_k}), &\mbox{if $1\leq j\leq p-1$,} \\ I(\mathbf{v}), &\mbox{if $j=p$.}\end{cases}
\end{equation}
Moreover, if $1\leq j\leq p-1$, then $\delta(v_{k-jn+q_k})=(1,1)$. In particular, the module $\tau_n^{-j}(P(\mathbf{v}))$ is indecomposable and not projective-injective.
\end{corollary}

\begin{proof}
We first prove (\ref{eq:j-th taun-inverse of P(v)}). For $j=0$ the result is clear. For $1\leq j\leq p-1$ we use induction on $j$, where the base case $j=1$ follows from Lemma \ref{lem:n-th cosyzygy and taun-inverse of P(v)}(b), while the induction step follows from Lemma \ref{lem:n-th cosyzygy and taun-inverse of simples in middle of k-flow paths}(b).

Next, if $1\leq j\leq p-1$, then
\[
2=2+1-1\leq n+1-q_1 = k-(p-1)n+q_k\leq k-jn+q_k \leq k-n+q_{k}\leq k-2+1 = k-1,
\]
from which it follows that $\delta(v_{k-jn+q_k})=(1,1)$ and so $S(v_{k-jn+q_k})$ is neither projective nor injective.

Finally, if $j=0$, then $\tau_n^{-j}(P(\mathbf{v}))\isom P(\mathbf{v})$ is not injective by Lemma \ref{lem:same P(v) implies same path}, while if $j=p$, then $\tau_n^{-j}(P(\mathbf{v}))\isom I(\mathbf{v})$ is not projective by the dual of Lemma \ref{lem:same P(v) implies same path}.
\end{proof}

\section{Sufficient conditions}\label{sec:sufficient conditions}

Motivated by Proposition \ref{prop:the cases A_1 and A_m tilde} and Proposition \ref{prop:n divides k+q(v)} we give the following definition.

\begin{definition}\label{def:n-admissible}
Let $Q$ be an $n$-pre-admissible quiver. We say that $Q$ is
\emph{$n$-admissible} if one of the following conditions hold:
\begin{enumerate}[label=(\alph*)]
    \item $Q=\tilde{A}_m$ and $n\divides m$, or
    \item $Q\neq \tilde{A}_m$ and for every $k$-flow path $\mathbf{v}$ in $Q$ we have that $n\divides (k+q(\mathbf{v}))$.
\end{enumerate}
\end{definition}

\begin{example}\label{ex:n-admissible quivers}
\begin{enumerate}[label=(\alph*)]
    \item The quiver $A_m$ is $n$-admissible if and only if $n\divides (m-1)$. In particular, the quiver $A_1$ is $n$-admissible for all $n\geq 2$.
    \item The quiver of Example \ref{ex:n-pre-admissible quivers}(c) is $2$-admissible.
    \item The quiver of Example \ref{ex:n-pre-admissible quivers}(d) is $3$-admissible but not $n$-admissible for any $n\neq 3$.
\end{enumerate}
\end{example}

\begin{remark}\label{rem:n-admissible for everything that divides n}
\begin{enumerate}[label=(\alph*)]
    \item When studying $n$-admissible quivers, the cases $Q=A_1$ and $Q=\tilde{A}_m$ for $m\geq 1$ usually behave differently from the rest of the cases; the reason for this is that the quivers $A_1$ and $\tilde{A}_m$ are the only $n$-pre-admissible quivers that do not have flow paths as Lemma \ref{lem:the existence of flow paths} shows. Hence many times in the rest of this paper we will exclude one or both of the cases $Q=A_1$ and $Q=\tilde{A}_m$ from our statements. We remind the reader that this does not present a problem in our aim of classification of $n$-cluster tilting subcategories for radical square zero bound quiver algebras since such a classification in these exceptional cases is given in Proposition \ref{prop:the cases A_1 and A_m tilde}.
    
    \item If $Q$ is an $n$-admissible quiver and $n'$ is an integer such that $n'\geq 2$ and $n'\divides n$, then it follows directly from Remark \ref{rem:n-pre-admissible is n-pre-admissible for any n} and Definition \ref{def:n-admissible} that $Q$ is also an $n'$-admissible quiver.
\end{enumerate}
\end{remark}

By Proposition \ref{prop:the cases A_1 and A_m tilde} and Proposition \ref{prop:n divides k+q(v)} it follows that if $Q$ is a quiver and there exists an $n$-cluster tilting subcategory $\cC\subseteq \m(\K Q/\J^2)$, then $Q$ is $n$-admissible. The aim of this section is to show that the opposite is also true. We also want to show that if $Q\neq \tilde{A}_m$, then $\cC$ is unique and give a description of $\cC$.

For the rest of this section we fix an $n$-admissible quiver $Q$ with $Q\neq A_1$ and $Q\neq \tilde{A}_m$ and we set $\La\coloneqq \K Q/\J^2$. We denote by $\mathbf{V}$ the set of all flow paths in $Q$. Note that by Lemma \ref{lem:the existence of flow paths} we have that $\mathbf{V}\neq \varnothing$. For a $k$-flow path $\mathbf{v}\in\mathbf{V}$ we set $p(\mathbf{v})=\frac{k+q(\mathbf{v})}{n}$; since $Q$ is $n$-admissible, it follows that $p(\mathbf{v})$ is an integer. We define
\[
M(\mathbf{v}) \coloneqq \bigoplus_{j=0}^{p(\mathbf{v})} \tau_n^{-j}(P(\mathbf{v})) \isom P(\mathbf{v})\oplus \left(\bigoplus_{j=1}^{p(\mathbf{v})-1}S(v_{k-jn+q_k})\right)\oplus I(\mathbf{v}),
\]
where the last isomorphism follows from Corollary \ref{cor:j-th taun-inverse of P(v)}. We also set $M(\mathbf{V})\coloneqq \bigoplus_{\mathbf{v}\in\mathbf{V}}M(\mathbf{v})$. With this notation we have the following lemmas.

\begin{lemma}\label{lem:M(V) is basic}
\begin{enumerate}[label=(\alph*)]
    \item The module $M(\mathbf{v})$ is basic and has no projective-injective direct summand.
    \item The module $M(\mathbf{V})$ is basic and has no projective-injective direct summand.
\end{enumerate}
\end{lemma}

\begin{proof}
\begin{enumerate}[label=(\alph*)]
    \item Follows immediately by Corollary \ref{cor:j-th taun-inverse of P(v)} and Lemma \ref{lem:same vertices in k-flow path}(a).
    
    \item By (a) we have that $M(\mathbf{V})$ has no projective-injective direct summand. It remains to show that $M(\mathbf{V})$ is basic. Since the module $M(\mathbf{v})$ for $\mathbf{v}\in\mathbf{V}$ is basic by (a), it is enough to show that if $\mathbf{v}$ and $\mathbf{u}$ are two flow paths in $Q$ with $\mathbf{v}\neq \mathbf{u}$, then $M(\mathbf{v})$ and $M(\mathbf{u})$ have no isomorphic direct summands. Assume towards a contradiction that there exist indecomposable direct summands $V$ of $M(\mathbf{v})$ and $U$ of $M(\mathbf{u})$ such that $V\isom U$ but $\mathbf{v}\neq \mathbf{u}$. Then $V\isom \tau_n^{-j_v}(P(\mathbf{v}))$ and $U\isom\tau_n^{-j_u}(P(\mathbf{u}))$ for some $j_v,j_u\in\ZZ_{\geq 0}$ with $j_v\leq p(\mathbf{v})$ and $j_u\leq p(\mathbf{u})$. Without loss of generality we assume that $j_u\geq j_v$. It follows that 
    \[
    \tau_n^{-(p(\mathbf{v})-j_v+j_u)}(P(\mathbf{u})) =
    \tau_n^{-(p(\mathbf{v})-j_v)}\tau_n^{-j_u}(P(\mathbf{u})) \isom \tau_n^{-(p(\mathbf{v})-j_v)}\tau_n^{-j_v}(P(\mathbf{v})) =
    \tau_n^{-p(\mathbf{v})}(P(\mathbf{v}))\isom I(\mathbf{v}),
    \]
    where the last isomorphism follows from Corollary \ref{cor:j-th taun-inverse of P(v)}. In particular, we have that the module $\tau_n^{-(p(\mathbf{v})-j_v+j_u)}(P(\mathbf{u}))$ is injective and nonzero. By Corollary \ref{cor:j-th taun-inverse of P(v)} we have that $\tau_n^{-j'}(P(\mathbf{u}))=0$ for $j'>p(\mathbf{u})$ and $\tau_n^{-j'}(P(\mathbf{u}))$ is not injective for $j'<p(\mathbf{u})$. We conclude that $p(\mathbf{v})-j_v+j_u=p(\mathbf{u})$ and so $I(\mathbf{u})\isom\tau_n^{-p(\mathbf{u})}(P(\mathbf{u}))=\tau_n^{-(p(\mathbf{v})-j_{v}+j_{u})}(P(\mathbf{u}))\isom I(\mathbf{v})$. Then by the dual of Lemma \ref{lem:same P(v) implies same path} it follows that $\mathbf{v}=\mathbf{u}$, which contradicts our assumption $\mathbf{v}\neq \mathbf{u}$. \qedhere 
\end{enumerate}
\end{proof}

\begin{lemma}\label{lem:M(V) is n-rigid}
Let $i\in\{1,\ldots,n-1\}$. Then $\Ext^{i}_{\La}(M(\mathbf{V}),M(\mathbf{V}))=0$.
\end{lemma}

\begin{proof}
    Let $\mathbf{v}$ be a $k$-flow path in $Q$ and let $\mathbf{u}$ be a $k'$-flow path in $Q$. By the definition of $M(\mathbf{V})$ and additivity of $\Ext_{\La}^{i}(-,-)$ it is enough to show that $\Ext^{i}_{\La}(M(\mathbf{u}),M(\mathbf{v}))=0$. By the definition of $M(\mathbf{u})$ and $M(\mathbf{v})$ and additivity of $\Ext^{i}_{\La}(-,-)$ it is enough to show that
    \begin{equation}\label{eq:no exts}
        \Ext^{i}_{\La}(\tau_n^{-x}(P(\mathbf{u})),\tau_n^{-y}(P(\mathbf{v})))=0
    \end{equation}
    for any $x\in\{0,1,\ldots,p(\mathbf{u})\}$ and $y\in\{0,1,\ldots,p(\mathbf{v})\}$. If $x=0$, then $\tau_n^{-x}(P(\mathbf{u}))=P(\mathbf{u})$ is projective and so (\ref{eq:no exts}) holds. If $y=p(\mathbf{v})$, then by Corollary \ref{cor:j-th taun-inverse of P(v)} we have that $\tau_n^{-p(\mathbf{v})}(P(\mathbf{v}))\isom I(\mathbf{v})$ is injective and so (\ref{eq:no exts}) holds again. Hence we may assume that $x>0$ and $y<p(\mathbf{v})$.
    
    Using dimension shift and the Aus\-lan\-der--Rei\-ten duality we compute
    \begin{align*}
        \Ext^{i}_{\La}(\tau_n^{-x}(P(\mathbf{u})),\tau_n^{-y}(P(\mathbf{v}))) &\isom \Ext^{1}_{\La}(\tau_n^{-x}(P(\mathbf{u})),\Omega^{-(i-1)}\tau_n^{-y}(P(\mathbf{v}))) \\
        &\isom D\underline{\Hom}_{\La}(\tau^{-}\Omega^{-(i-1)}\tau_n^{-y}(P(\mathbf{v})), \tau_n^{-x}(P(\mathbf{u}))) \\
        &\isom D\underline{\Hom}_{\La}(\tau_i^{-}\tau_n^{-y}(P(\mathbf{v})), \tau_n^{-x}(P(\mathbf{u}))) \\
        &\isom D\underline{\Hom}_{\La}(S(v_{k-yn-i+q_k(\mathbf{v})}), \tau_n^{-x}(P(\mathbf{u}))),
    \end{align*}
    where the last isomorphism follows from Lemma \ref{lem:n-th cosyzygy and taun-inverse of P(v)}(b) if $y=0$ and by Corollary \ref{cor:j-th taun-inverse of P(v)} and Lemma \ref{lem:n-th cosyzygy and taun-inverse of simples in middle of k-flow paths}(b) if $y>0$. Hence it is enough to show that
    \begin{equation}\label{eq:no nonzero exts}
        D\underline{\Hom}_{\La}(S(v_{k-yn-i+q_k(\mathbf{v})}), \tau_n^{-x}(P(\mathbf{u}))) = 0.
    \end{equation}
    Assume towards a contradiction that (\ref{eq:no nonzero exts}) does not hold. We consider the cases $0<x<p(\mathbf{u})$ and $x=p(\mathbf{u})$ separately and reach a contradiction in each case.
    
    Case $0<x<p(\mathbf{u})$. In this case by Corollary \ref{cor:j-th taun-inverse of P(v)} we have that $\tau_n^{-x}(P(\mathbf{u}))\isom S(u_{k'-xn+q_{k'}(\mathbf{u})})$. Then it follows that $\Hom_{\La}(S(v_{k-yn-i+q_k(\mathbf{v})}),S(u_{k'-xn+q_{k'}(\mathbf{u})}))\neq 0$. Since both modules are simple, we conclude that $v_{k-yn-i+q_k(\mathbf{v})}=u_{k'-xn+q_{k'}(\mathbf{u})}$. By Corollary \ref{cor:j-th taun-inverse of P(v)} and since $0<x<p(\mathbf{u})$, it follows that $\delta(u_{k'-xn+q_{k'}(\mathbf{u})})=(1,1)$. Thus by Lemma \ref{lem:same flow paths}(a) we obtain $\mathbf{v}=\mathbf{u}$. In particular, we have that $k=k'$ and $q_k(\mathbf{v})=q_{k'}(\mathbf{u})$ and so $v_{k-yn-i+q_k(\mathbf{v})}=v_{k-xn+q_k(\mathbf{v})}$. Hence by Lemma \ref{lem:same vertices in k-flow path}(a) it follows that $k-yn-i+q_k(\mathbf{v})=k-xn+q_k(\mathbf{v})$. Equivalently we get $(x-y)n=i$, which contradicts $1\leq i \leq n-1$.
    
    Case $x=p(\mathbf{u})$. In this case by Corollary \ref{cor:j-th taun-inverse of P(v)} we have that $\tau_n^{-x}(P(\mathbf{u}))\isom I(\mathbf{u})$. Since we assume that (\ref{eq:no nonzero exts}) does not hold, and since $I(\mathbf{u})$ is indecomposable and injective, it follows that $S(v_{k-yn-i+q_k(\mathbf{v})})\isom\soc(I(\mathbf{u}))$. We consider the subcases $\delta^{+}(u_{1})=1$ and $\delta^{+}(u_{1})=2$ separately.
    \begin{itemize}
        \item Subcase $\delta^{+}(u_{1})=1$. In this case we have $I(\mathbf{u})=I(u_1)$ by definition. Hence $v_{k-yn-i+q_k(\mathbf{v})}=u_1$ and so $\delta(v_{k-yn-i+q_k(\mathbf{v})})\neq (1,1)$. By the definition of a $k$-flow path we obtain that $k-yn-i+q_k(\mathbf{v})\in\{1,k\}$. We claim that $k-yn-i+q_k(\mathbf{v})=1$. Indeed, assume towards a contradiction that $k-yn-i+q_k(\mathbf{v})=k$. Since $0\leq y \leq p(\mathbf{v})-1$, $1\leq i\leq n-1$ and $0\leq q_k(\mathbf{v}) \leq 1$, it follows that $y=0$, $i=1$ and $q_k(\mathbf{v})=1$. But then $(1,2)=\delta(v_k)=\delta(v_{k-yn-i+q_k(\mathbf{v})})=\delta(u_1)$ contradicts the fact that $\delta^{+}(u_{1})=1$. 
        
        Hence we have $k-yn-i+q_k(\mathbf{v})=1$. Using this equality together with $k+q(\mathbf{v})=p(\mathbf{v})n$, we obtain that $(p(\mathbf{v})-y)n=i+q_1(\mathbf{v})$. Since $y<p(\mathbf{v})$ and $1\leq i \leq n-1$, it follows that $q_1(\mathbf{v})=1$. Hence we have $v_1=v_{k-yn-i+q_k(\mathbf{v})}=u_{1}$ and $\delta(v_1)=(2,1)$. Then any morphism from $S(v_{k-yn-i+q_k(\mathbf{v})})=S(v_1)=\begin{smallmatrix} v_1 \end{smallmatrix}$ to $\tau_n^{-x}(P(\mathbf{u}))\isom I(u_1)=I(v_1)=\begin{smallmatrix} v^{-2}\;\;v^{-3} \\ v_1 \end{smallmatrix}$ clearly factors through $P(v^{-2})=\begin{smallmatrix} \;\;\;v^{-2} \\ v_1 \end{smallmatrix}$. But this shows that (\ref{eq:no nonzero exts}) holds, which is a contradiction.
        
        \item Subcase $\delta^{+}(u_{1})=2$. In this case we have $I(\mathbf{u})=I(u^{-1})$ by definition. Hence $v_{k-yn-i+q_k(\mathbf{v})}=u^{-1}$. Then any morphism from $S(v_{k-yn-i+q_k(\mathbf{v})})=S(u^{-1})=\begin{smallmatrix} u^{-1} \end{smallmatrix}$ to $\tau_n^{-x}(P(\mathbf{u}))\isom I(u^{-1})=\begin{smallmatrix} u_1 \\ \;\;u^{-1} \end{smallmatrix}$ clearly factors through $P(u_1)= \begin{smallmatrix} u_1 \\ u^{-1}\;\;u_2\end{smallmatrix}$. But this shows that (\ref{eq:no nonzero exts}) holds, which is a contradiction. \qedhere
    \end{itemize}
\end{proof}

\begin{lemma}\label{lem:vertices with degree (2,2) have no ext-1s}
Let $v,u\in Q_0$ be such that $\delta(v)=\delta(u)=(2,2)$.
\begin{enumerate}[label=(\alph*)]
    \item We have $\Ext^{1}_{\La}(M(\mathbf{V}),P(v))=0$ and $\Ext^1_{\La}(I(v),M(\mathbf{V}))=0$.
    \item We have $\Ext^{1}_{\La}(I(u),P(v))=0$.
\end{enumerate}
\end{lemma}

\begin{proof}
\begin{enumerate}[label=(\alph*)]
    \item We only show that $\Ext^1_{\La}(M(\mathbf{V}),P(v))=0$; the other equality follows dually. Let $\mathbf{w}$ be a $k$-flow path in $Q$. By additivity of $\Ext^{1}_{\La}(-,-)$ it is enough to show that
    \[
    \Ext^{1}_{\La}(\tau_n^{-x}(P(\mathbf{w})),P(v))=0
    \]
    for any $x\in\{0,1,\ldots,p(\mathbf{w})\}$. If $x=0$, then $\tau_n^{-x}(P(\mathbf{w}))=P(\mathbf{w})$ is projective and so the result follows. Otherwise, assume that $1\leq x\leq p(\mathbf{w})$. By the dual of Lemma \ref{lem:cosyzygy and tau-inverse of simples}(c) we have that $\tau^{-}(P(v))\isom S(v)$. Then by the Aus\-lan\-der--Rei\-ten duality, it is enough to show that
    \begin{equation}\label{eq:ext-1 for the case n=2}
        D\underline{\Hom}_{\La}(S(v), \tau_n^{-x}(P(\mathbf{w})))=0.
    \end{equation}
    We consider the cases $1\leq x\leq p(\mathbf{w})-1$ and $x=p(\mathbf{w})$ separately.
    
    Case $1\leq x\leq p(\mathbf{w})-1$. In this case by Corollary \ref{cor:j-th taun-inverse of P(v)} we have that $\tau_n^{-x}(P(\mathbf{w}))\isom S(w_{k-xn+q_k})$. Assume towards a contradiction that (\ref{eq:ext-1 for the case n=2}) does not hold. Then $S(v)\isom S(w_{k-xn+q_k})$ from which it follows that $v=w_{k-xn+q_k}$. By Corollary \ref{cor:j-th taun-inverse of P(v)} we have that $\delta(w_{k-xn+q_k})=(1,1)$, which contradicts $\delta(v)=(2,2)$.
    
    Case $x=p(\mathbf{w})$. In this case by Corollary \ref{cor:j-th taun-inverse of P(v)} we have that $\tau_n^{-x}(P(\mathbf{w}))\isom I(\mathbf{w})$. Assume towards a contradiction that (\ref{eq:ext-1 for the case n=2}) does not hold. Then $S(v)\isom \soc (I(\mathbf{w}))$ from which it follows that $I(v)\isom I(\mathbf{w})$. But this contradicts the dual of Lemma \ref{lem:indecomposable projectives noninjectives come from flow paths, are injective or have degree (2,2)} since $\delta(v)=(2,2)$.
    
    \item By the dual of Lemma \ref{lem:cosyzygy and tau-inverse of simples}(c) we have that $\tau^{-}(P(v))\isom S(v)$. Then by the Aus\-lan\-der--Rei\-ten duality it is enough to show that
    \[
    D\underline{\Hom}_{\La}(S(v),I(u))=0.
    \]
    If $v\neq u$, then $\Hom_{\La}(S(v),I(u))=0$ and the result follows. Otherwise, assume that $v=u$. Let $w_1\longrightarrow v$ and $w_2\longrightarrow v$ be the arrows ending at $v$. Then any morphism from $S(v)=\begin{smallmatrix}
    v
    \end{smallmatrix}$ to $I(u)=I(v)=\begin{smallmatrix}
    w_1\;\;w_2 \\ v
    \end{smallmatrix}$ clearly factors through $P(w_1)=\begin{smallmatrix}
    \;w_1 \\ v
    \end{smallmatrix}$, which shows that $D\underline{\Hom}_{\La}(S(v),I(u))=0$.\qedhere
\end{enumerate}
\end{proof}

Next, let $\{R_t\}_{t=1}^{f}$ be a complete collection of representatives of pairwise non-isomorphic projective-injective $\La$-modules. Set 
\begin{equation}\label{eq:definition of M}
M \coloneqq M(\mathbf{V}) \oplus \left(\bigoplus_{t=1}^{f}R_t\right)\oplus \left(\bigoplus_{\substack{v\in Q_0 \\ \delta(v)=(2,2)}} (P(v)\oplus I(v))\right).
\end{equation}
The main aim of this section is to show that $M$ is the unique $n$-cluster tilting module of $\La$. We start by giving an alternate description of $M$.

\begin{corollary}\label{cor:M is basic and tau_n inverse of Lambda}
The module $M$ is basic and $M\isom \bigoplus_{j\geq 0}\tau_n^{-j}(\La)$. In particular, we have that $D(\La)\in M$.
\end{corollary}

\begin{proof}
We set 
\[
R\coloneqq \bigoplus_{t=1}^{f}R_t, \text{ and } M_{(2,2)}\coloneqq \bigoplus_{\substack{v\in Q_0 \\ \delta(v)=(2,2)}} (P(v)\oplus I(v)).
\]
By Lemma \ref{lem:indecomposable projectives noninjectives come from flow paths, are injective or have degree (2,2)} we have that
\[
\La \isom \left(\bigoplus_{\mathbf{v}\in\mathbf{V}}P(\mathbf{v})\right)\oplus R \oplus \left(\bigoplus_{\substack{v\in Q_0 \\ \delta(v)=(2,2)}} P(v)\right).
\]
Then by (\ref{eq:definition of M}) and Corollary \ref{cor:taun-inverse of P(v) for degree (2,2)} it follows that $M\isom \bigoplus_{j\geq 0}\tau_n^{-j}(\La)$. 

To see that $M$ is basic, we have that $M(\mathbf{V})$ is basic by Lemma \ref{lem:M(V) is basic}(b), that $R$ is basic by definition and that $M_{(2,2)}$ is basic since $P(v)$ is never injective if $\delta^{+}(v)=2$. By Corollary \ref{cor:j-th taun-inverse of P(v)} and by Lemma \ref{lem:indecomposable projectives noninjectives come from flow paths, are injective or have degree (2,2)} and its dual and by comparing direct summands of $M(\mathbf{V})$, $R$ and $M_{(2,2)}$, it easily follows that $M$ is basic.

Finally, we show that $D(\La)\in\add(M)$. It is enough to show that for every vertex $v\in Q_0$, the indecomposable injective $\La$-module $I(v)$ corresponding to the vertex $v\in Q_0$ belongs to $\add(M)$. If $\delta(v)=(2,2)$ or $I(v)$ is projective, then clearly $I(v)\in\add(M)$ by the definition of $M$. Otherwise, by the dual of Lemma \ref{lem:indecomposable projectives noninjectives come from flow paths, are injective or have degree (2,2)} it follows that $I(v)\isom I(\mathbf{v})$ for some flow path $\mathbf{v}$ in $Q$. Then by Corollary \ref{cor:taun-inverse of P(v) for degree (2,2)} and Proposition \ref{prop:n-ct is closed under n-AR translations and uniqueness of n-ct}(a) we have
\[
I(v) \isom \tau_n^{-p(\mathbf{v})}(P(\mathbf{v})) \in \add(M),
\]
as required.
\end{proof}

Next we want to show that $M$ is $n$-rigid.

\begin{proposition}\label{prop:M is n-rigid}
Let $i\in\{1,\ldots,n-1\}$. Then $\Ext^{i}_{\La}(M,M)=0$.
\end{proposition}

\begin{proof}
By Lemma \ref{lem:M(V) is n-rigid} and since $R_t$ is projective-injective for every $t\in\{1,\ldots,f\}$, it follows that the module $M(\mathbf{V}) \oplus \left(\bigoplus_{t=1}^{f}R_t\right)$ is $n$-rigid. Hence if there exists no vertex $v\in Q_0$ with degree $\delta(v)=(2,2)$, the result follows immediately, while if there exists a vertex $v\in Q_0$ with degree $\delta(v)=(2,2)$, the result follows from Lemma \ref{lem:vertices with degree (2,2) have no ext-1s}.
\end{proof}

We are now ready to show that $M$ is $n$-cluster tilting.

\begin{proposition}\label{prop:n-admissible gives n-cluster tilting}
The module $M$ is an $n$-cluster tilting $\La$-module and any basic $n$-cluster tilting $\La$-module is isomorphic to $M$.
\end{proposition}

\begin{proof}
To show that $M$ is an $n$-cluster tilting module we need to show that
\begin{align*}
    \add(M) &= \{X\in\m\La \mid \Ext^{i}_{\La}(M,X)=0\text{ for all $0 < i < n$}\} \\
    &= \{X\in\m\La \mid \Ext^{i}_{\La}(X,M)=0\text{ for all $0 < i < n$}\}.
\end{align*}
We only show the first equality; the other follows dually. Since by Proposition \ref{prop:M is n-rigid} the module $M$ is $n$-rigid, the inclusion
\[
\add(M) \subseteq \{X\in\m\La \mid \Ext^{i}_{\La}(M,X)=0\text{ for all $0 < i < n$}\}
\]
holds. It remains to show the opposite inclusion, that is that if $\Ext^{i}_{\La}(M,X)=0$ for all $0<i<n$, then $X\in\add(M)$. We show the contrapositive statement that if $X\not\in \add(M)$, then $\Ext^{i}_{\La}(M,X)\neq 0$ for some $i\in\{1,\dots,n-1\}$. By additivity of $\Ext^{i}_{\La}(-,-)$ we may assume that $X$ is indecomposable. Since by Corollary \ref{cor:M is basic and tau_n inverse of Lambda} we have that $\La\in\add(M)$ and $D(\La)\in\add(M)$, it follows that $X$ is neither projective nor injective. Since $X$ is neither projective nor injective, it follows from Theorem \ref{thrm:representation-finite string algebra}(d) that $X$ is simple. Then $X\isom S(v)$ for some vertex $v\in Q_0$. Clearly $\delta(v)\neq (0,1)$ and $\delta(v)\neq (1,0)$ because in the first case we have that $S(v)$ is injective while in the second case we have that $S(v)$ is projective. We consider the cases $\delta^{-}(v)=2$ and $\delta^{-}(v)=1$ separately.

Case $\delta^{-}(v)=2$. In this case we have by Lemma \ref{lem:cosyzygy and tau-inverse of simples}(c) that $\tau^{-}(S(v))\isom I(v)$. By Lemma \ref{lem:taun-inverse gives nonzero ext} it follows that $\Ext^{1}_{\La}(I(v),S(v))\neq 0$. Since by Corollary \ref{cor:M is basic and tau_n inverse of Lambda} we have that $I(v)\in\add(M)$, we conclude that $\Ext^{1}_{\La}(M,S(v))\neq 0$, as required.

Case $\delta^{-}(v)=1$. In this case we have that $\delta(v)=(1,1)$ or $\delta(v)=(1,2)$. By Corollary \ref{cor:unique flow path} there exists a unique $k$-flow path $\mathbf{v}$ in $Q$ such that $v=v_j$ for some $j>1$. Notice that $j<k+q_k$ also holds. We first claim that $n$ does not divide $k-j+q_k$. 

To show this, assume towards a contradiction that $k-j+q_k=mn$ for some $m\in\ZZ_{\geq 0}$. Then $j=k-mn+q_k$. Since we have $1<j<k+q_k$, we obtain 
\[
1 < k-mn+q_k<k+q_k.
\]
Using $k+q(\mathbf{v})=p(\mathbf{v})n$ and $q(\mathbf{v})=-1+q_1+q_k$, we obtain that
\[
0<m<p(\mathbf{v})-\frac{q_1}{n},
\]
which implies $0<m<p(\mathbf{v})$. But then by Corollary \ref{cor:j-th taun-inverse of P(v)} we have
\[
X\isom S(v)=S(v_j) = S(v_{k-mn+q_k}) \isom \tau_n^{-m}(P(\mathbf{v}))\in\add(M),
\]
which contradicts $X\not\in \add(M)$.

Hence $n$ does not divide $k-j+q_k$. Let $m$ be the unique integer such that $m<\frac{k-j+q_k}{n}<m+1$. Using $1<j<k+q_k$, we obtain that $0\leq m \leq p(\mathbf{v})-1$. Then by Lemma \ref{lem:n-th cosyzygy and taun-inverse of P(v)}, Corollary \ref{cor:j-th taun-inverse of P(v)} and Lemma \ref{lem:n-th cosyzygy and taun-inverse of simples in middle of k-flow paths} it follows that
\begin{equation}\label{eq:omega and taun of P}
\Omega^{-(k-mn+q_k-j)}\tau_n^{-m}(P(\mathbf{v})) \isom S(v_j).
\end{equation}
Set $i\coloneqq (m+1)n-k-q_k+j$. Since $m<\frac{k-j+q_k}{n}<m+1$, we obtain that $0<i<n$. Then, using (\ref{eq:omega and taun of P}), we compute
\begin{align*}
    \tau_{i}^{-}(S(v_j)) &\isom \tau_{i}^{-} \Omega^{-(k-mn+q_k-j)}\tau_n^{-m}(P(\mathbf{v})) \\
    &= \tau^{-}\Omega^{-(i-1+k-mn+q_k-j)}\tau_n^{-m}(P(\mathbf{v})) \\
    &= \tau^{-}\Omega^{-((m+1)n-k-q_k+j-1+k-mn+q_k-j)}\tau_n^{-m}(P(\mathbf{v})) \\
    &=\tau^{-}\Omega^{-(n-1)}\tau_n^{-m}(P(\mathbf{v})) \\
    &= \tau^{-(m+1)}_n(P(\mathbf{v})).
\end{align*}
By Corollary \ref{cor:j-th taun-inverse of P(v)} and since $0\leq m\leq p(\mathbf{v})-1$, it follows that $\tau_n^{-(m+1)}(P(\mathbf{v}))\neq 0$. Then by Lemma \ref{lem:taun-inverse gives nonzero ext} we have that $\Ext^{i}_{\La}(\tau_n^{-(m+1)}(P(\mathbf{v})),S(v_j))\neq 0$, which shows that $\Ext^{i}_{\La}(M,S(v))\neq 0$ since $\tau_n^{-(m+1)}(P(\mathbf{v}))\in\add(M)$.

Finally, the fact that $M$ is the unique basic $n$-cluster tilting module up to isomorphism follows from Proposition \ref{prop:n-ct is closed under n-AR translations and uniqueness of n-ct}(c).
\end{proof}

\section{Main result and applications}\label{sec:main result and applications}

We are now ready to state our main result.

\begin{theorem}\label{thrm:n-cluster tilting iff n-admissible}
Let $Q$ be a quiver, let $\La=\K Q/\J^2$ and let $n\in\ZZ_{\geq 2}$. Then the algebra $\La$ admits an $n$-cluster tilting subcategory $\cC\subseteq \m\La$ if and only if $Q$ is an $n$-admissible quiver. If moreover $Q\neq \tilde{A}_m$ for any $m\geq 1$, then $\cC$ is unique and $\cC=\add\left(\bigoplus_{j\geq 0}\tau_n^{-j}(\La)\right)$.
\end{theorem}

\begin{proof}
The statement that if $\La$ admits an $n$-cluster tilting subcategory, then $Q$ is an $n$-admissible quiver follows from Proposition \ref{prop:Q is n-pre-admissible}, Proposition \ref{prop:the cases A_1 and A_m tilde} and Proposition \ref{prop:n divides k+q(v)}. The statement that if $Q$ is an $n$-admissible quiver, then $\La$ admits an $n$-cluster tilting subcategory follows from Proposition \ref{prop:the cases A_1 and A_m tilde} and Proposition \ref{prop:n-admissible gives n-cluster tilting}. The description of $\cC$ in the case $Q\neq \tilde{A}_m$ follows from Proposition \ref{prop:n-admissible gives n-cluster tilting}.
\end{proof}

\begin{remark}\label{rem:1-cluster tilting subcategories}
In Theorem \ref{thrm:n-cluster tilting iff n-admissible} we classify $n$-cluster tilting subcategories for bound quiver algebras of the form $\K Q/\J^2$ when $n\geq 2$. We also find that all of them are of the form $\add(M)$ for an $n$-cluster tilting module $M$. If $n=1$, then the algebra $\La=\K Q/\J^2$ admits a unique $1$-cluster tilting subcategory, namely the whole module category $\m \La$. Moreover, the module category $\m\La$ is of the form $\add(M)$ if and only if $\La$ is a rep\-re\-sen\-ta\-tion-fi\-nite algebra. A result of Gabriel \cite{Gab1} classifies rep\-re\-sen\-ta\-tion-fi\-nite algebras with radical square zero in terms of their \emph{separated quiver}; see also \cite[Section X.2]{ARS}.
\end{remark}

Using Theorem \ref{thrm:n-cluster tilting iff n-admissible} we can construct many examples of algebras that admit $n$-cluster tilting modules and have many interesting properties. As an example, an answer to a question of Erdmann and Holm from \cite{EH} is given in \cite{MV} using radical square zero bound quiver algebras.

\begin{example}\label{ex:n-cluster tilting for n-admissible}
\begin{enumerate}[label=(\alph*)]
    \item Let $Q=A_m$, let $\La=\K Q/\J^2$ and let $n\geq 2$ be such that $n\divides (m-1)$. Then \[
    \La\oplus \left(\bigoplus_{j=1}^{\frac{m-1}{n}}S(m-jn)\right)
    \]
    is the unique basic $n$-cluster tilting $\La$-module.
    
    \item Let $Q$ be as in Example \ref{ex:n-pre-admissible quivers}(c) and let $\La=\K Q/\J^2$. Then the module
    \begin{align*}
        M &= \La\oplus\tau_2^{-}(\La)\oplus \tau_2^{-2}(\La) \\
        &\isom \La \oplus \left(\begin{smallmatrix}
        1 \\ 2\;\; 
        \end{smallmatrix}\oplus \begin{smallmatrix}
        7
        \end{smallmatrix}\oplus \begin{smallmatrix}
        2\;\;8 \\ 3
        \end{smallmatrix}\oplus \begin{smallmatrix}
        4 \;\; 5 \\ 5
        \end{smallmatrix}\oplus \begin{smallmatrix}
        3 \\ \;\; 6
        \end{smallmatrix}\oplus \begin{smallmatrix}
        1 \\ \;\; 1
        \end{smallmatrix}\right) \oplus \begin{smallmatrix}
        3 \\ 4\;\;
        \end{smallmatrix}
    \end{align*}
    is the unique basic $2$-cluster tilting $\La$-module.
    
    \item Let $Q$ be as in Example \ref{ex:n-pre-admissible quivers}(d) and let $\La=\K Q/\J^2$. Then the module 
    \begin{align*}
        M &= \La\oplus\tau_3^{-}(\La) \\
        &\isom \La\oplus\left(\begin{smallmatrix}
        10 \;\; 7\; \\ 8
        \end{smallmatrix} \oplus \begin{smallmatrix}
        5 \\ \;\;6
        \end{smallmatrix} \oplus \begin{smallmatrix}
        \;3 \;\;12 \\ 4
        \end{smallmatrix} \oplus \begin{smallmatrix}
        1 \\ \;\;2
        \end{smallmatrix} \oplus \begin{smallmatrix}
        5 \\ 11\;\;\;
        \end{smallmatrix} \oplus
        \begin{smallmatrix}
        1 \\ 9\;\;
        \end{smallmatrix}\right)
    \end{align*}
    is the unique basic $3$-cluster tilting $\La$-module.
\end{enumerate}
\end{example}

In the rest of this section we further investigate some properties of radical square zero bound quiver algebras which admit $n$-cluster tilting subcategories. We start with describing a method to construct $n$-admissible quivers.

\begin{remark}\label{rem:how to find n-admissible quivers}
Starting from any $n$-pre-admissible quiver $Q$, it is not difficult to construct an $n$-admissible quiver by adjusting the lengths of flow paths in $Q$ appropriately. For example, if $Q$ is the quiver
\[
\begin{tikzcd}[row sep=small, column sep=small]
    1 \arrow[r] \arrow[loop above] & 2,
\end{tikzcd}
\]
then $Q$ is $n$-pre-admissible for any $n\geq 2$ and there are two flow paths in $Q$, namely
\begin{align*}
    \mathbf{v}&: 1\longrightarrow 1, \\
    \mathbf{u}&: 1\longrightarrow 2.
\end{align*}
In particular, we have $q(\mathbf{v})=0$ and $q(\mathbf{u})=-1$. Now let us fix an $n\geq 2$ and construct an $n$-admissible quiver. We pick $k_{\mathbf{v}},k_{\mathbf{u}}\geq 2$ such that $n\divides k_{\mathbf{v}}$ and $n\divides (k_{\mathbf{u}}-1)$. Then the quiver
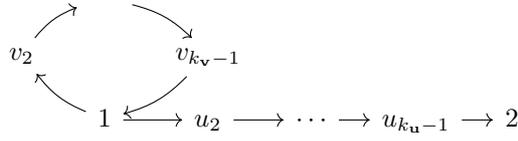
\[
\begin{tikzcd}[row sep=small, column sep=small]
    {} & \cdots \arrow[bend left=15, rd] & {} & {} & {} & {} \\
    v_2 \arrow[bend left=15, ru] & {} & v_{k_{\mathbf{v}}-1} \arrow[bend left=15, ld] & {} & {} & {}\\
    {} & 1 \arrow[bend left=15, lu] \arrow[r] & u_2 \arrow[r] & \cdots \arrow[r] & u_{k_{\mathbf{u}}-1} \arrow[r] & 2
\end{tikzcd}
\]
is $n$-admissible.
\end{remark}

\subsection{\texorpdfstring{$n\ZZ$}{nZ}-cluster tilting subcategories}

We recall the definition of $n\ZZ$-cluster tilting subcategories from \cite{IJ}.

\begin{definition}\cite[Definition-Proposition 2.15]{IJ}\label{def:nZ-cluster-tilting}
Let $\La$ be an algebra and let $\cC\subseteq\m\La$ be an $n$-cluster tilting subcategory. We say that $\cC$ is an \emph{$n\ZZ$-cluster tilting subcategory} if one of the two equivalent conditions
\begin{enumerate}[label=(\alph*)]
    \item $\Omega^{n}(\cC)\subseteq \cC$, and
    \item $\Omega^{-n}(\cC)\subseteq \cC$
\end{enumerate}
holds.
\end{definition}

In this subsection we classify radical square zero bound quiver algebras which admit $n\ZZ$-cluster tilting subcategories. We start with the following proposition.

\begin{proposition}\label{prop:degree of vertex in nZ-ct}
Let $\La=\K Q/\J^2$ and assume that there exists an $n\ZZ$-cluster tilting subcategory $\cC\subseteq\m\La$. Let $v\in Q_0$ be a vertex of $Q$. Then $\delta(v)\in\{(0,0),(0,1),(1,0),(1,1)\}$.
\end{proposition}

\begin{proof}
Since $\cC$ is an $n\ZZ$-cluster tilting subcategory, it follows that $Q$ is an $n$-admissible quiver. Hence $\delta^{+}(v)\leq 2$ and $\delta^{-}(v)\leq 2$ and it is enough to show that $\delta^{+}(v)\neq 2$ and $\delta^{-}(v)\neq 2$. We show that $\delta^{+}(v)\neq 2$; the fact that $\delta^{-}(v)\neq 2$ follows dually.

Assume towards a contradiction that $\delta^{+}(v)=2$ and let $v\longrightarrow u_1$ and $v\longrightarrow u_2$ be the two arrows starting at $v$. Then $P(v)\in \cC$ and $P(v)$ is not injective. It follows from Proposition \ref{prop:n-ct is closed under n-AR translations and uniqueness of n-ct}(a) that $P(v)\isom \tau_n(X)$ for some nonprojective indecomposable module $X\in\cC$. In particular, the module $\Omega^{n-1}(X)$ is nonprojective and so we have
\[
\Omega^{n-1}(X) \isom \tau^{-}\tau\Omega^{n-1}(X)= \tau^{-}\tau_n(X) \isom \tau^{-}(P(v)) \isom S(v),
\]
where the last isomorphism follows from the dual of Lemma \ref{lem:cosyzygy and tau-inverse of simples}(c). Since by the dual of Lemma \ref{lem:cosyzygy and tau-inverse of simples}(c) we have $\Omega(S(v))\isom S(u_1)\oplus S(u_2)$, we obtain that
\[
\Omega^{n}(X)= \Omega\Omega^{n-1}(X)\isom \Omega(S(v))\isom S(u_1)\oplus S(u_2).
\]
Since $\cC$ is an $n\ZZ$-cluster tilting subcategory, it follows that $S(u_1)\oplus S(u_2)\in\cC$. But then a direct computation shows that $\Omega(I(u_2))\isom S(u_1)$, from which we conclude that $\Ext^{1}_{\La}(I(u_2),S(u_1)\oplus S(u_2))\neq 0$. This contradicts the fact that $\cC$ is an $n$-cluster tilting subcategory since $I(u_2),S(u_1)\oplus S(u_2)\in\cC$. 
\end{proof}

We can now give the classification of $n\ZZ$-cluster tilting subcategories for radical square zero bound quiver algebras.

\begin{theorem}\label{thrm:nZ-cluster tilting iff Nakayama}
Let $Q$ be a quiver, let $\La=\K Q/\J^2$ and let $n\in\ZZ_{\geq 2}$. Then the algebra $\La$ admits an $n\ZZ$-cluster tilting subcategory $\cC\subseteq \m\La$ if and only if $Q=A_m$ and $n\divides (m-1)$ or $Q=\tilde{A}_m$ and $n\divides m$.
\end{theorem}

\begin{proof}
If $Q=A_m$ and $n\divides (m-1)$ or $Q=\tilde{A}_m$ and $n\divides m$, then $\La$ admits an $n$-cluster tilting subcategory $\cC\subseteq\m\La$ by Theorem \ref{thrm:n-cluster tilting iff n-admissible}. Moreover, in this case, it is easy to see that $\tau(M)\isom \Omega(M)$ for any $M\in\m\La$ and hence $\cC$ is also an $n\ZZ$-cluster tilting subcategory by Proposition \ref{prop:n-ct is closed under n-AR translations and uniqueness of n-ct}(a).

For the other direction, assume that $\La$ admits an $n\ZZ$-cluster tilting subcategory $\cC$. Then by Proposition \ref{prop:degree of vertex in nZ-ct} we have that if $v\in Q_0$, then $\delta(v)\in\{(0,0),(1,0),(0,1),(1,1)\}$. Since $Q$ is connected, we conclude that there exists some $m\in\ZZ_{\geq 1}$ such that $Q=A_m$ or $Q=\tilde{A}_m$. Since $\cC$ is $n$-cluster tilting, it follows that $Q$ is $n$-admissible from Theorem \ref{thrm:n-cluster tilting iff n-admissible}. Hence we conclude that if $Q=A_m$, then $n\divides (m-1)$, while if $Q=\tilde{A}_m$, then $n\divides m$, as required. 
\end{proof}

In particular we see that the only radical square zero bound quiver algebras which admit $n\ZZ$-cluster tilting subcategories are Nakayama algebras.

\subsection{A lattice of \texorpdfstring{$n$}{n}-cluster tilting subcategories}
Before giving our next result, let us recall the following classical definition.

\begin{definition}\label{def:lattice}
A \emph{poset} is a partially ordered set. A \emph{lattice}
is a partially ordered set in which every two elements have a \emph{meet}, that is a greatest lower bound and a \emph{join}, that is a least upper bound. A \emph{complete lattice} is a lattice in which any subset has a greatest lower bound and a least upper bound.
\end{definition}

\begin{example}\label{ex:complete lattice of divisors}
Let $N$ be a positive integer. Then the set $D(N)=\{x\in\ZZ \mid x\geq 1 \text{ and } x \divides N\}$ forms a complete lattice called the \emph{lattice of divisors of $N$} under the relation $x \leq y$ if $x\divides y$. If $x,y\in D(N)$, then their meet corresponds to their greatest common divisor $\gcd(x,y)$ and their join corresponds to their least common multiple $\lcm(x,y)$.
\end{example}

For the rest of this article, we drop our assumption that we consider $n$-cluster tilting subcategories for $n\geq 2$ and we assume that $n\geq 1$ instead. Let $Q\neq \tilde{A}_m$ be a quiver and let $\La=\K Q/\J^2$. Our aim is to show that the collection of $n$-cluster tilting subcategories (for varying $n$) of $\m\La$ forms a lattice with respect to inclusion of subcategories. We start with the following result.

\begin{proposition}\label{prop:inclusion of subcategories}
Let $\La=\K Q/\J^2$ be a radical zero bound quiver algebra and assume that $Q\neq A_1$ and $Q\neq \tilde{A}_m$ for any $m\geq 1$. Let $\cC_n\subseteq \m\La$ be an $n$-cluster tilting subcategory and $\cC_{n'}\subseteq\m\La$ be an $n'$-cluster tilting subcategory. Then $\cC_n\subseteq\cC_{n'}$ if and only if $n'\divides n$.
\end{proposition}

\begin{proof}
If $n'=1$, then the result is clear since $\cC_{n'}=\cC_1=\m\La$. If $n=1$, then the result is also clear since $\m\La$ is an $n$-cluster tilting subcategory if and only if $n=1$ (since $Q\neq A_1$). 

Hence we may assume that $n>1$ and $n'>1$. Then $Q$ is $n$-admissible and $n'$-admissible by Theorem \ref{thrm:n-cluster tilting iff n-admissible}. Moreover, we have that $\cC_n=\add(M_n)$ and $\cC_{n'}=\add(M_{n'})$ where
\[
M_n = \bigoplus_{j\geq 0}\tau_n^{-j}(\La) \text{ and } M_{n'} = \bigoplus_{j\geq 0}\tau_{n'}^{-j}(\La).
\]

Assume first that $n'\divides n$. Then $n=hn'$ for some $h\geq 1$. Let $X\in\cC_n$ and we show that $X\in\cC_{n'}$. Since $\cC_n$ and $\cC_{n'}$ are closed under direct sums and summands, we may assume that $X$ is indecomposable. If $X$ is projective or injective, then $X\in\cC_{n'}$ since $\cC_{n'}$ is an $n'$-cluster tilting subcategory. Otherwise we have by (\ref{eq:definition of M}) and Corollary \ref{cor:M is basic and tau_n inverse of Lambda} that $X\isom \tau_n^{-j}(P(\mathbf{v}))$ for some $k$-flow path $\mathbf{v}$ in $Q$ and some $j\geq 1$. Since $Q$ is $n$-admissible and $n'$-admissible, we have that $k+q(\mathbf{v})=pn$ and $k+q(\mathbf{v})=p'n'$ for some $p,p'\in\ZZ_{\geq 1}$. In particular, we have $p=\frac{p'n'}{n}$. Moreover, by Corollary \ref{cor:j-th taun-inverse of P(v)} and since $X$ is not injective, we have that $1\leq j\leq p-1$. Hence we obtain
\[
1 \leq jh \leq (p-1)h = ph - h = \frac{p'n'}{n}\frac{n}{n'}-h=p'-h\leq p'-1,
\]
and so $1\leq jh\leq p'-1$. Hence by Corollary \ref{cor:j-th taun-inverse of P(v)} we have
\[
X\isom \tau_n^{-j}(P(\mathbf{v}))\isom S(v_{k-jn+q_k})=S(v_{k-jhn'+q_k})\isom\tau_{n'}^{-jh}(P(\mathbf{v}))\in\add(M_{n'})=\cC_{n'},
\]
as required.

Assume now that $\cC_n\subseteq\cC_{n'}$. Then by Lemma \ref{lem:the existence of flow paths} there exists a $k$-flow path $\mathbf{v}$ in $Q$. Since $Q$ is $n$-admissible and $n'$-admissible, we have that $k+q(\mathbf{v})=pn$ and $k+q(\mathbf{v})=p'n'$ for some $p,p'\in\ZZ_{\geq 1}$. If $p=1$, then $n=p'n'$ and $n'\divides n$ as required. Otherwise, assume that $p>1$. Then by Corollary \ref{cor:j-th taun-inverse of P(v)} and Proposition \ref{prop:n-ct is closed under n-AR translations and uniqueness of n-ct}(a) we have that
\[
\tau_n^{-}(P(\mathbf{v})) \isom S(v_{k-n+q_{k}})\in\cC_n.
\]
Since by assumption we have $\cC_n\subseteq\cC_{n'}$, we conclude that $S(v_{k-n+q_{k}})\in\cC_{n'}$. Write $n=hn'+r$ with $h\in \ZZ_{\geq 0}$ and $0\leq r\leq n'-1$. We first claim that $1\leq h \leq p'-1$. 

First assume towards a contradiction that $h=0$. Then $n=r$ and
\[
k-n+q_k=(p-1)n+1-q_1\geq (2-1)\cdot 1+1-1 = 1.
\]
Hence by Lemma \ref{lem:n-th cosyzygy and taun-inverse of P(v)}(a) we have that $\Omega^{-n}(P(\mathbf{v}))\isom S(v_{k-n+q_k})$. Since $0<n=r<n'$, we have $0<n'-n<n'$. Hence by Proposition \ref{prop:n-ct is closed under n-AR translations and uniqueness of n-ct}(a) we obtain 
\[
\tau_{n'}^{-}(P(\mathbf{v})) = \tau^{-}_{n'-n} \Omega^{-n}(P(\mathbf{v})) \isom \tau^{-}_{n'-n}(S(v_{k-n+q_k}))\in\cC_{n'}. 
\]
It follows from Lemma \ref{lem:taun-inverse gives nonzero ext} that $\Ext^{n'-n}_{\La}(\tau_{n'}^{-}(P(\mathbf{v})),S(v_{k-n+q_k}))\neq 0$. This contradicts the fact that $\cC_{n'}$ is $n'$-cluster tilting since $\tau_{n'}^{-}(P(\mathbf{v})),S(v_{k-n+q_k})\in\cC_{n'}$.

Next assume towards a contradiction that $h\geq p'$. Then we have
\[
n < pn = k+q(\mathbf{v}) = p'n' \leq hn' \leq hn'+r = n,
\]
which is a contradiction.

We conclude that $n=hn'+r$ with $1\leq h\leq p'-1$ and we now claim that $r=0$. Assume towards a contradiction that $r>0$. By Corollary \ref{cor:j-th taun-inverse of P(v)} and Proposition \ref{prop:n-ct is closed under n-AR translations and uniqueness of n-ct}(a) we have that 
\[
\tau_{n'}^{-h}(P(\mathbf{v})) \isom S(v_{k-hn'+q_k})\in \cC_{n'}.
\]
Then $1\leq k-hn'+q_k\leq k-1+q_k$ and $(k-hn'+q_k) -r = k-(hn'+r)+q_k = k-n+q_k \geq 1$, and so by Lemma \ref{lem:n-th cosyzygy and taun-inverse of simples in middle of k-flow paths}(a) we have that
\[
\Omega^{-r}(S(v_{k-hn'+q_k})) \isom S(v_{k-hn'-r+q_k})=S(v_{k-n+q_k}).
\]
But then we have that
\[
\Ext^r_{\La}(S(v_{k-n+q_k}),S(v_{k-hn'+q_k})) \isom  \Ext^r_{\La}(\Omega^{-r}(S(v_{k-hn'+q_k})),S(v_{k-hn'+q_k}))\neq 0.
\]
This contradicts the fact that $\cC_{n'}$ is $n'$-cluster tilting since $S(v_{k-n+q_k}), S(v_{k-hn'+q_k})\in\cC_{n'}$ and $1\leq r\leq n'-1$. We conclude that $r=0$ and so $n=hn'$, as required.
\end{proof}

We also need the following definition.

\begin{definition}\label{def:admissible degree}
Let $Q$ be a quiver. We define the \emph{admissible degree} of $Q$ to be
\[
N(Q) \coloneqq \begin{cases} \max\left(\{n\in \ZZ_{\geq 2} \mid \text{$Q$ is $n$-admissible}\}\cup\{1\}\right), &\mbox{if $Q\neq A_1$,} \\ 1, &\mbox{if $Q=A_1$.} \end{cases} 
\]
\end{definition}

Since $Q$ is finite, it follows that $N(Q)$ is well-defined. We now give the main result for this section.

\begin{theorem}\label{thrm:lattice of n-ct}
Let $Q$ be a quiver with admissible degree $N=N(Q)$ and let $D(N)=\{n\in\ZZ \mid n\geq 1 \text{ and } n\divides N\}$. Let $\La=\K Q/\J^2$.
\begin{enumerate}[label=(\alph*)]
    \item If $Q\neq A_1$, then there exists an $n$-cluster tilting subcategory $\cC_n\subseteq\m\La$ if and only if $n\in D(N)$.
    
    \item If $Q\neq \tilde{A}_m$, set
    \[
    \mathbf{CT}(\La)\coloneqq\{\cC\subseteq \m\La \mid \text{there exists $n\in\ZZ_{\geq 1}$ such that $\cC$ is $n$-cluster tilting}\}.
    \]
    Then for every $n\in D(N)$ there exists a unique $n$-cluster tilting subcategory $\cC_n$. Moreover, the pair $(\mathbf{CT}(\La),\subseteq)$ is a poset isomorphic to the opposite of the poset of divisors of $N$. In particular, $(\mathbf{CT}(\La),\subseteq)$ forms a complete lattice where the meet of $\cC_n$ and $\cC_{n'}$ is given by $\cC_{\lcm(n,n')}$ and the join of $\cC_n$ and $\cC_{n'}$ is given by $\cC_{\gcd(n,n')}$.
\end{enumerate}
\end{theorem}

\begin{proof}
\begin{enumerate}[label=(\alph*)]
    \item For $n=1$ the result is clear since $\m\La$ is a $1$-cluster tilting subcategory. Assume now that $n>1$. If $n\in D(N)$, then it follows from Remark \ref{rem:n-admissible for everything that divides n}(b) and Theorem \ref{thrm:n-cluster tilting iff n-admissible} that $\La$ admits an $n$-cluster tilting subcategory, which proves one direction.
    
    For the other direction assume that there exists an $n$-cluster tilting subcategory $\cC_{n}\subseteq\m\La$ and we show that $n\in D(N)$. It follows from Theorem \ref{thrm:n-cluster tilting iff n-admissible} that $Q$ is $n$-admissible and so $1<n\leq N$. Hence by Definition \ref{def:admissible degree} it follows that $Q$ is $N$-admissible too. We consider the cases $Q=\tilde{A}_m$ and $Q\neq \tilde{A}_m$ separately.
    
    If $Q=\tilde{A}_m$ for some $m\geq 1$, then we have by Definition \ref{def:n-admissible} that $N\divides m$ and so $N\leq m$. Moreover, the quiver $\tilde{A}_m$ is always $m$-admissible and so $m\leq N$. It follows that $m=N$. Since $Q$ is also $n$-admissible, we have that $n\divides m=N$ and so $n\in D(N)$.
    
    If $Q\neq\tilde{A}_m$, then there exists a flow path $\mathbf{v}$ in $Q$ by Lemma \ref{lem:the existence of flow paths}. Moreover, for every $k_{\mathbf{v}}$-flow path $\mathbf{v}$ we have that $n\divides (k_{\mathbf{v}}+q(\mathbf{v}))$ and $N\divides (k_{\mathbf{v}}+q(\mathbf{v}))$. It follows that $\lcm(n,N)\divides (k_{\mathbf{v}}+q(\mathbf{v}))$ for every flow path $\mathbf{v}$ in $Q$. Hence $Q$ is $\lcm(n,N)$-admissible and so $\lcm(n,N)\leq N$. We conclude that $n\divides N$ and so $n\in D(N)$. 
    
    \item If $N(Q)=1$, then $\mathbf{CT}(\La)=\{\m\La\}$ and the result is clear. If $N(Q)>1$, then existence of $\cC_n$ follows from (a) and uniqueness by Theorem \ref{thrm:n-cluster tilting iff n-admissible}. Then $(\mathbf{CT}(\La),\subseteq)$ is a poset isomorphic to the opposite of the poset of divisors of $N$ by Proposition \ref{prop:inclusion of subcategories}. That $\mathbf{CT}(\La)$ forms a complete lattice with the given meet and join follows from Example \ref{ex:complete lattice of divisors}.\qedhere
\end{enumerate}
\end{proof}

We finish with an example which illustrates Theorem \ref{thrm:lattice of n-ct}.

\begin{example}\label{ex:lattice of n-ct}
Let $Q$ be the quiver
\[
\begin{tikzcd}[row sep=small, column sep=small]
    {} & 23 \arrow[ld] & 22 \arrow[l] & 21 \arrow[l] & 20 \arrow[l] & 19 \arrow[l] \\
    1 \arrow[r] \arrow[d] & 14 \arrow[r] & 15 \arrow[r] & 16 \arrow[r] & 17 \arrow[r] & 18 \arrow[u] \\
    2 \arrow[r] & 3 \arrow[r] & 4 \arrow[r] & 5 \arrow[r] & 6 \arrow[r] & 7 \arrow[r] & 8 \arrow[r] & 9 \arrow[r] & 10 \arrow[r] & 11 \arrow[r] & 12 \arrow[r] & 13.
\end{tikzcd}
\]
Then $N(Q)=12$. The Aus\-lan\-der--Rei\-ten quiver of $\La=\K Q/\J^2$ is
\[
\begin{tikzpicture}[scale=0.87, transform shape, baseline={(current bounding box.center)}]
\tikzstyle{nct}=[circle, minimum width=0.6cm, draw, inner sep=0pt, text centered, scale=0.9]
\tikzstyle{nout}=[circle, minimum width=6pt, draw=none, inner sep=0pt, scale=0.9]
    
    \node[nout] (13) at (0,2.1) {$\begin{smallmatrix}
    13
    \end{smallmatrix}$};
    \node[nout] (12 13) at (0.7,2.8) {$\begin{smallmatrix}
    12 \\ 13
    \end{smallmatrix}$};
    \node[nout] (12) at (1.4,2.1) {$\begin{smallmatrix}
    12
    \end{smallmatrix}$};
    \node[nout] (11 12) at (2.1,2.8) {$\begin{smallmatrix}
    11 \\ 12
    \end{smallmatrix}$};
    \node[nout] (11) at (2.8, 2.1) {$\begin{smallmatrix}
    11
    \end{smallmatrix}$};
    \node[nout] (10 11) at (3.5,2.8) {$\begin{smallmatrix}
    10 \\ 11
    \end{smallmatrix}$};
    \node[nout] (10) at (4.2,2.1) {$\begin{smallmatrix}
    10
    \end{smallmatrix}$};
    \node[nout] (9 10) at (4.9,2.8) {$\begin{smallmatrix}
    9 \\ 10
    \end{smallmatrix}$};
    \node[nout] (9) at (5.6,2.1) {$\begin{smallmatrix}
    9
    \end{smallmatrix}$};
    \node[nout] (8 9) at (6.3,2.8) {$\begin{smallmatrix}
    8 \\ 9
    \end{smallmatrix}$};
    \node[nout] (8) at (7,2.1) {$\begin{smallmatrix}
    8
    \end{smallmatrix}$};
    \node[nout] (7 8) at (7.7,2.8) {$\begin{smallmatrix}
    7 \\ 8
    \end{smallmatrix}$};
    \node[nout] (7) at (8.4,2.1) {$\begin{smallmatrix}
    7
    \end{smallmatrix}$};
    \node[nout] (6 7) at (9.1,2.8) {$\begin{smallmatrix}
    6 \\ 7
    \end{smallmatrix}$};
    \node[nout] (6) at (9.8,2.1) {$\begin{smallmatrix}
    6
    \end{smallmatrix}$};
    \node[nout] (5 6) at (10.5,2.8) {$\begin{smallmatrix}
    5 \\ 6
    \end{smallmatrix}$};
    \node[nout] (5) at (11.2,2.1) {$\begin{smallmatrix}
    5
    \end{smallmatrix}$};
    \node[nout] (4 5) at (11.9,2.8) {$\begin{smallmatrix}
    4 \\ 5
    \end{smallmatrix}$};
    \node[nout] (4) at (12.6,2.1) {$\begin{smallmatrix}
    4
    \end{smallmatrix}$};
    \node[nout] (3 4) at (13.3,2.8) {$\begin{smallmatrix}
    3 \\ 4
    \end{smallmatrix}$};
    \node[nout] (3) at (14,2.1) {$\begin{smallmatrix}
    3
    \end{smallmatrix}$};
    \node[nout] (2 3) at (14.7,2.8) {$\begin{smallmatrix}
    2 \\ 3
    \end{smallmatrix}$};
    \node[nout] (2) at (15.4,2.1) {$\begin{smallmatrix}
    2
    \end{smallmatrix}$};
    
    \node[nout] (1) at (1.4,0.7) {$\begin{smallmatrix}
    1
    \end{smallmatrix}$};
    \node[nout] (23 1) at (2.1,0) {$\begin{smallmatrix}
    23 \\ 1
    \end{smallmatrix}$};
    \node[nout] (23) at (2.8,0.7) {$\begin{smallmatrix}
    23
    \end{smallmatrix}$};
    \node[nout] (22 23) at (3.5,0) {$\begin{smallmatrix}
    22 \\ 23
    \end{smallmatrix}$};
    \node[nout] (22) at (4.2,0.7) {$\begin{smallmatrix}
    22
    \end{smallmatrix}$};
    \node[nout] (21 22) at (4.9,0) {$\begin{smallmatrix}
    21 \\ 22
    \end{smallmatrix}$};
    \node[nout] (21) at (5.6,0.7) {$\begin{smallmatrix}
    21
    \end{smallmatrix}$};
    \node[nout] (20 21) at (6.3,0) {$\begin{smallmatrix}
    20 \\ 21
    \end{smallmatrix}$};
    \node[nout] (20) at (7,0.7) {$\begin{smallmatrix}
    20
    \end{smallmatrix}$};
    \node[nout] (19 20) at (7.7,0) {$\begin{smallmatrix}
    19 \\ 20
    \end{smallmatrix}$};
    \node[nout] (19) at (8.4,0.7) {$\begin{smallmatrix}
    19
    \end{smallmatrix}$};
    \node[nout] (18 19) at (9.1,0) {$\begin{smallmatrix}
    18 \\ 19
    \end{smallmatrix}$};
    \node[nout] (18) at (9.8,0.7) {$\begin{smallmatrix}
    18
    \end{smallmatrix}$};
    \node[nout] (17 18) at (10.5,0) {$\begin{smallmatrix}
    17 \\ 18
    \end{smallmatrix}$};
    \node[nout] (17) at (11.2,0.7) {$\begin{smallmatrix}
    17
    \end{smallmatrix}$};
    \node[nout] (16 17) at (11.9,0) {$\begin{smallmatrix}
    16 \\ 17
    \end{smallmatrix}$};
    \node[nout] (16) at (12.6,0.7) {$\begin{smallmatrix}
    16
    \end{smallmatrix}$};
    \node[nout] (15 16) at (13.3,0) {$\begin{smallmatrix}
    15 \\ 16
    \end{smallmatrix}$};
    \node[nout] (15) at (14,0.7) {$\begin{smallmatrix}
    15
    \end{smallmatrix}$};
    \node[nout] (14 15) at (14.7,0) {$\begin{smallmatrix}
    14 \\ 15
    \end{smallmatrix}$};
    \node[nout] (14) at (15.4,0.7) {$\begin{smallmatrix}
    14
    \end{smallmatrix}$};
    
    \node[nout] (1 2 14) at (16.1,1.4) {$\begin{smallmatrix}
    1 \\ 2\;\;14
    \end{smallmatrix}$};
    \node[nout] (1 14) at (16.8,2.1) {$\begin{smallmatrix}
    1 \\ \;\;\;14
    \end{smallmatrix}$};
    \node[nout] (1 2) at (16.8,0.7) {$\begin{smallmatrix}
    \;\;\;1 \\ 2
    \end{smallmatrix}$};
    \node[nout] (1b) at (17.5,1.4) {$\begin{smallmatrix}
    1.
    \end{smallmatrix}$};
    
    \draw[->] (13) to (12 13);
    \draw[->] (12 13) to (12);
    \draw[->] (12) to (11 12);
    \draw[->] (11 12) to (11);
    \draw[->] (11) to (10 11);
    \draw[->] (10 11) to (10);
    \draw[->] (10) to (9 10);
    \draw[->] (9 10) to (9);
    \draw[->] (9) to (8 9);
    \draw[->] (8 9) to (8);
    \draw[->] (8) to (7 8);
    \draw[->] (7 8) to (7);
    \draw[->] (7) to (6 7);
    \draw[->] (6 7) to (6);
    \draw[->] (6) to (5 6);
    \draw[->] (5 6) to (5);
    \draw[->] (5) to (4 5);
    \draw[->] (4 5) to (4);
    \draw[->] (4) to (3 4);
    \draw[->] (3 4) to (3);
    \draw[->] (3) to (2 3);
    \draw[->] (2 3) to (2);
    
    \draw[->] (1) to (23 1);
    \draw[->] (23 1) to (23);
    \draw[->] (23) to (22 23);
    \draw[->] (22 23) to (22);
    \draw[->] (22) to (21 22);
    \draw[->] (21 22) to (21);
    \draw[->] (21) to (20 21);
    \draw[->] (20 21) to (20);
    \draw[->] (20) to (19 20);
    \draw[->] (19 20) to (19);
    \draw[->] (19) to (18 19);
    \draw[->] (18 19) to (18);
    \draw[->] (18) to (17 18);
    \draw[->] (17 18) to (17);
    \draw[->] (17) to (16 17);
    \draw[->] (16 17) to (16);
    \draw[->] (16) to (15 16);
    \draw[->] (15 16) to (15);
    \draw[->] (15) to (14 15);
    \draw[->] (14 15) to (14);
    
    \draw[->] (2) to (1 2 14);
    \draw[->] (14) to (1 2 14);
    \draw[->] (1 2 14) to (1 14);
    \draw[->] (1 2 14) to (1 2);
    \draw[->] (1 14) to (1b);
    \draw[->] (1 2) to (1b);
    
    \draw[loosely dotted] (13.east) -- (12);
    \draw[loosely dotted] (12.east) -- (11);
    \draw[loosely dotted] (11.east) -- (10);
    \draw[loosely dotted] (10.east) -- (9);
    \draw[loosely dotted] (9.east) -- (8);
    \draw[loosely dotted] (8.east) -- (7);
    \draw[loosely dotted] (7.east) -- (6);
    \draw[loosely dotted] (6.east) -- (5);
    \draw[loosely dotted] (5.east) -- (4);
    \draw[loosely dotted] (4.east) -- (3);
    \draw[loosely dotted] (3.east) -- (2);
    
    \draw[loosely dotted] (1.east) -- (23);
    \draw[loosely dotted] (23.east) -- (22);
    \draw[loosely dotted] (22.east) -- (21);
    \draw[loosely dotted] (21.east) -- (20);
    \draw[loosely dotted] (20.east) -- (19);
    \draw[loosely dotted] (19.east) -- (18);
    \draw[loosely dotted] (18.east) -- (17);
    \draw[loosely dotted] (17.east) -- (16);
    \draw[loosely dotted] (16.east) -- (15);
    \draw[loosely dotted] (15.east) -- (14);
    
    \draw[loosely dotted] (2.east) -- (1 14);
    \draw[loosely dotted] (14.east) -- (1 2);
    \draw[loosely dotted] (1 2 14.east) -- (1b);
\end{tikzpicture}
\]
The divisors of $12$ are $D(12)=\{1,2,3,4,6,12\}$. For $n\in D(12)$ we set $M_n=\bigoplus_{j\geq 0}\tau_n^{-j}(\La)$ and $\cC_n=\add(M_n)$. Then we have
\begin{align*}
    &\cC_{1}=\m\La,
    && \cC_{2}=\add\{\La, \begin{smallmatrix}
    11
    \end{smallmatrix}, \begin{smallmatrix}
    9
    \end{smallmatrix}, \begin{smallmatrix}
    7
    \end{smallmatrix}, \begin{smallmatrix}
    5
    \end{smallmatrix}, \begin{smallmatrix}
    3
    \end{smallmatrix}, \begin{smallmatrix}
    1 \\ 14
    \end{smallmatrix}, \begin{smallmatrix}
    23
    \end{smallmatrix}, \begin{smallmatrix}
    21
    \end{smallmatrix}, \begin{smallmatrix}
    19
    \end{smallmatrix}, \begin{smallmatrix}
    17
    \end{smallmatrix}, \begin{smallmatrix}
    15
    \end{smallmatrix}, \begin{smallmatrix}
    1 \\ 2
    \end{smallmatrix}\}, \\
    & \cC_{3}=\add\{\La, \begin{smallmatrix}
    10
    \end{smallmatrix}, \begin{smallmatrix}
    7
    \end{smallmatrix}, \begin{smallmatrix}
    4
    \end{smallmatrix}, \begin{smallmatrix}
    1 \\ 14
    \end{smallmatrix}, \begin{smallmatrix}
    22
    \end{smallmatrix}, \begin{smallmatrix}
    19
    \end{smallmatrix}, \begin{smallmatrix}
    16
    \end{smallmatrix}, \begin{smallmatrix}
    1 \\ 2
    \end{smallmatrix}\},
    && \cC_{4}=\add\{\La,\begin{smallmatrix}
    9
    \end{smallmatrix}, \begin{smallmatrix}
    5
    \end{smallmatrix}, \begin{smallmatrix}
    1 \\ 14
    \end{smallmatrix}, \begin{smallmatrix}
    21
    \end{smallmatrix}, \begin{smallmatrix}
    17
    \end{smallmatrix}, \begin{smallmatrix}
    1 \\ 2
    \end{smallmatrix}\}, \\
    & \cC_{6}=\add\{\La, \begin{smallmatrix}
    7
    \end{smallmatrix}, \begin{smallmatrix}
    1 \\ 14
    \end{smallmatrix}, \begin{smallmatrix}
    19
    \end{smallmatrix}, \begin{smallmatrix}
    1 \\ 2
    \end{smallmatrix}\},
    && \cC_{12} = \add\{\La,\begin{smallmatrix}
    1 \\ 14
    \end{smallmatrix},\begin{smallmatrix}
    1 \\ 2
    \end{smallmatrix}\}, 
\end{align*}
and $\cC_n$ is an $n$-cluster tilting subcategory of $\m\La$ by Theorem \ref{thrm:n-cluster tilting iff n-admissible}. Then the lattice 
\[
\begin{tikzcd}[column sep=.7em, row sep=.7em]
    & 12 \arrow[dr, dash] & & \\
    4 \arrow[ru, dash] \arrow[rd, dash] & & 6 \arrow[rd, dash] \\
    & 2 \arrow[ru, dash] \arrow[rd, dash] & & 3 \\
    & & 1 \arrow[ru, dash] & 
\end{tikzcd}
\]
of divisors of $12$ corresponds to the lattice
\[
\begin{tikzcd}[column sep=.7em, row sep=.7em]
    & \cC_{12} \arrow[dr, symbol=\subset] & & \\
    \cC_{4} \arrow[ru, symbol=\supset] \arrow[rd, symbol=\subset] & & \cC_{6} \arrow[rd, symbol=\subset] \\
    & \cC_{2} \arrow[ru, symbol=\supset] \arrow[rd, symbol=\subset] & & \cC_{3} \\
    & & \cC_{1} \arrow[ru, symbol=\supset] & 
\end{tikzcd}
\]
of inclusions of $n$-cluster tilting subcategories of $\m\La$.
\end{example}

\section*{Acknowledgements}
The author wishes to thank Hipolito Treffinger for valuable discussions on string algebras and helpful suggestions about the article. The author also thanks Ren{\'e} Marczinzik for bringing to his attention the question in paragraph 5.4 of \cite{EH}. The author is grateful to Max Planck Institute for Mathematics in Bonn for its hospitality and financial support.

\end{document}